\documentclass[12pt]{article}
\usepackage{natbib}

\usepackage{hyperref}

\usepackage{a4wide}
\usepackage{amscd}
\usepackage{enumerate}

\usepackage{graphicx}

\usepackage{amsmath,multirow}
\usepackage{array}
\usepackage{mathptmx}
\usepackage[utf8]{inputenc}
\usepackage{mathtools}
\usepackage{amsfonts}
\usepackage{ed}
\usepackage[mathscr]{eucal}

\usepackage{amsthm}

\theoremstyle{plain}
\newtheorem{theorem}{Theorem}[section]
\newtheorem{corollary}{Corollary}[section]
\newtheorem{lemma}{Lemma}[section]
\newtheorem{remark}{Remark}

\newtheorem{definition}{Definition}[section]

\numberwithin{equation}{section}

\newcommand{\rom}[1]{\uppercase\expandafter{\romannumeral #1\relax}}
\begin{document}
{
  \title{\bf Type $\rom{1}$, $\rom{2}$, $\rom{3}$ and $\rom{4}$ $q$-negative binomial distributions of order $k$}
 \author{Jungtaek Oh\thanks{Corresponding Author:
 Department of Biomedical Science, Kyungpook National University, Daegue, 41566, Republic of Korea
 (e-mail: jungtaekoh0191@gmail.com)}\hspace{.2cm}\\
}
  \maketitle
} 

\begin{abstract}
The distributions of waiting times in variations of the negative binomial distribution of order $k$. In one variation, we apply different enumeration schemes on the runs of successes. In another variation, binary trials with a geometrically varying probability of ones were performed. The exact distribution of the waiting time for the $r$-th occurrence of a success run of a specified length (e.g., nonoverlapping, overlapping, at least, exact, and $\ell$-overlapping) in a $q$-sequence of binary trials. First, the waiting time for the $r$-th occurrence of a success run with the "nonoverlapping" counting scheme was examined. Theorem \ref{pmftype1q-negative} gives a probability function of the Type \rom{1} $q$-negative binomial distribution of order $k$. Next, we consider the waiting time for the $r$-th occurrence of a success run with the "at least" counting scheme. Theorem \ref{thmtype2q-nega}  gives a probability function of the Type \rom{2} $q$-negative binomial distribution of order $k$. Next, we consider the waiting time for the $r$-th occurrence of a success run with the "overlapping" counting scheme. Theorem \ref{thmtype3q-nega} gives a probability function of the Type \rom{3} $q$-negative binomial distribution of order $k$. Next, we consider the waiting time for the $r$-th occurrence of a success run with the "exact" counting scheme. Theorem \ref{thmtype4q-nega} gives a probability function of the Type \rom{4} $q$-negative binomial distribution of order $k$. Next, we consider the waiting time for the $r$-th occurrence of a success run with the "$\ell$-overlapping" counting scheme, which is a more generalized counting scheme. Theorem \ref{pmfl-overq-nega} gives a probability function of the $q$-negative binomial distribution of order $k$ in the $\ell$-overlapping case.

The main theorems are Type \rom{1}, \rom{2}, \rom{3}, and \rom{4} $q$-negative binomial distributions of order $k$ and the $q$-negative binomial distribution of order $k$ in the $\ell$-overlapping case. This work examined sequences of independent binary zero and one trials with distributions that are not necessarily identical and with the probability of ones varying according to a geometric rule. The exact formulas for the distributions were obtained using enumerative combinatorics.

\end{abstract}

\noindent%
\vfill

%



\section{Introduction}
\citet{charalambides2010a} studied discrete $q$-distributions on Bernoulli trials with geometrically varying success probabilities. Let us consider a sequence $X_{1}$,...,$X_{n}$ of zero(failure)-one(success) Bernoulli trials, such that the trials of the subsequence after the $(i-1)$-th zero until the $i$-th zero are independent with equal failure probability. The $i$-th geometric sequences of trials is the subsequence after the $(i-1)$-th zero and until
the $i$-th zero, for $i>0$, and the subsequence after the $(j-1)$-th zero and until
the $j$-th zero, for $j>0$ are independent for all $i\neq j$ (i.e., the $i$-th and $j$-th geometric sequences are independent) with probability of zeroes at the $i$-th geometric sequence of trials.
\noindent
\begin{equation}\label{failureprobofq}
\begin{split}
q_{i}=1-\theta q^{i-1},\quad i=1,2,..., \quad 0\leq\theta\leq1,\quad 0\leq q<1.
\end{split}
\end{equation}
The probability of failures in the independent geometric sequences of trials increases geometrically with rate $q$.
Let $F_{j}=\Sigma_{m=1}^{j}(1-X_{m})$ denote the number of zeroes in the first $j$ trials. Because the probability of a zero at the $i$-th geometric sequence of trials is in fact the conditional probability of the occurrence of a zero at any trial $j$ given the occurrence of $(i-1)$ zeroes in the previous trials. We can express this situation as follows.
\noindent
\begin{equation}\label{failureprobofq2}
\begin{split}
q_{j,i}=p\Big(X_{j}=0\ \bigm|\ F_{j}=i-1\Big)=1-\theta q^{i-1},\quad i=1,2,...,j, \quad j=1,2,....
\end{split}
\end{equation}
\noindent
Here, \eqref{failureprobofq} is exactly equal to the conditional probability given in \eqref{failureprobofq2}. For a clear understanding, we consider an example with $n = 18$, i.e., the binary sequence 111011110111100110. Each subsequence has its own success and failure probabilities according to a geometric rule.

\begin{figure}[h]
 \begin{center}
\includegraphics[scale=0.65]{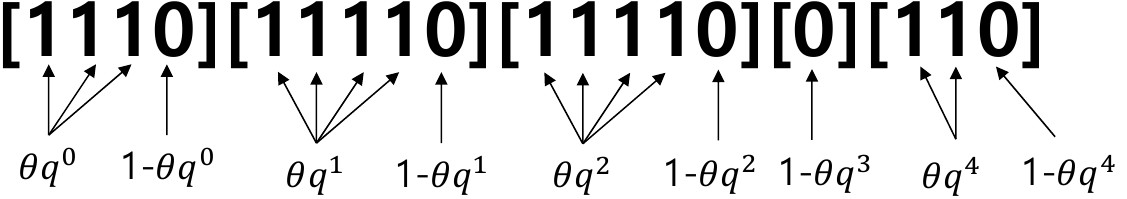}
\end{center}
  \vspace{-20pt}
\end{figure}

The stochastic model \eqref{failureprobofq} or \eqref{failureprobofq2} has interesting applications. It was studied as a reliability growth model by \cite{Dubman&Sherman1969} and applied to a $q$-boson theory in physics by \cite{Jing&Fan1994} and \cite{Jing1994}. More specifically, the $q$-binomial distribution was introduced as a $q$-deformed binomial distribution, to set up a $q$-binomial state. The stochastic model \eqref{failureprobofq} was also applied as a sequential-intervention model to start-up demonstration tests, as proposed by \cite{BBV1995}.\\

The stochastic model \eqref{failureprobofq} is the $q$-analog of the classical binomial distribution with geometrically varying probability of zeroes, which is a stochastic model of independent and identically distributed (IID) trials with the following failure probability:
\noindent
\begin{equation}\label{bernoullifailureprob}
\pi_{j}=P\left(X_{j}=0\right)=1-\theta,\ j=1,2,\ldots,\ 0<\theta<1.
\end{equation}
\noindent
As $q$ tends toward $1$, the stochastic model \eqref{failureprobofq} reduces to the IID (Bernoulli) model \eqref{bernoullifailureprob}, because $q_{i}\rightarrow\pi_{i}$, $i=1,2,\ldots$ or $q_{j,i}\rightarrow1-\theta$, $i=1,2,\ldots,j$, $j=1,2,\ldots.$\\

Discrete $q$-distributions based on the stochastic model of the sequence of independent Bernoulli trials have been explored by numerous researchers. For a lucid review and comprehensive list of publications in this area, the interested reader may consult the monographs by \cite{Charalambides2010b,charalambides2010a,Charalambides2016}.\\

From mathematical and statistical viewpoints, \cite{Charalambides2016}, in the preface of his book remarked upon the advantage of using a stochastic model of a sequence of independent Bernoulli trials, in which the probability of success in a trial is assumed to vary with the number of trials and/or the number of successes. Such a model is advantageous because it allows us to incorporate the experience gained from previous trials and/or successes. If the probability of success at a trial is a general function of the number of trials and/or the number of successes, little can be inferred about the distributions of the various random variables that can possibly be defined using this model. The assumption that the probability of success (or failure) at a trial varies geometrically, with the rate (proportion) $q$, leads to the introduction of discrete $q$-distributions.\\

Let us consider a distribution that is related to the success run analog of the classical negative binomial distribution. Let $X_{1},X_{2},\ldots $ be a sequence of binary trials with two possible outcomes (i.e., success or failure) in each trial. In this work, we examined the waiting time until the $r$-th (where $r$ is a positive integer) appearance of a success run of length $k$ by considering the enumeration scheme (i.e., nonoverlapping, at least, overlapping, exact, or $\ell$-overlapping scheme). Notably, the special case $r=1$ reduces to the geometric distribution of order $k$ (the distribution of the number of trials until the success run of length $k$, denoted as $T_{k}$). When considering the waiting distribution, different counting schemes are used, and each scheme generates a different type of waiting time distribution.\\

There are several ways to count a scheme. Each counting scheme depends on different conditions: whether overlapping counting is permitted and whether counting starts from scratch when a certain kind or size of run has been enumerated. \cite{feller1968introduction} proposed a classical counting method: once $k$ consecutive successes are observed, the number of occurrences of $k$ consecutive successes is counted, and the counting procedure starts anew (i.e., from scratch). This is called a \textit{nonoverlapping} counting scheme or Type $\rom{1}$ distributions of order $k$. In another scheme, success runs of length greater than or equal to $k$ preceded and followed by a failure or by the beginning or by the end of the sequence are counted(see. e.g., \citet{mood1940distribution}), and this is usually called the \textit{at least} counting scheme or Type $\rom{2}$ distribution of order $k$. \citet{ling1988binomial} suggested the \textit{overlapping} counting scheme, wherein an uninterrupted sequence of $m\geq k$ successes preceded and followed by a failure or by the beginning or by the end of the sequence are counted. It accounts for $m-k+1$ success runs of length of $k$ and is also referred to as Type $\rom{3}$ distributions of order $k$. \citet{mood1940distribution} suggested an \textit{exact} counting scheme, wherein success runs of length exactly $k$ preceded and succeeded by failure or by nothing are counted. This scheme is also referred to as Type $\rom{4}$ distributions of order $k$.\\
It is well known that the negative binomial distribution arises as the distribution of the sum of $r$ independent random variables
 following a geometric distribution with parameter $p$. The random variable $W_{r,k}^{(a)}$ denoted by the waiting time for the $r$-th occurrence of a success run with the counting scheme used $a=\rom{1}$, which indicates the nonoverlapping counting scheme; $a=\rom{2}$, which indicates the at least counting scheme; $a=\rom{3}$, which indicates the overlapping one; and $a=\rom{4}$, which indicates the exactly one these schemes are denoted as $W_{r,k}^{(\rom{1})}$, $W_{r,k}^{(\rom{2})}$, $W_{r,k}^{(\rom{3})}$, and $W_{r,k}^{(\rom{4})}$, respectively.  In addition, if the sequence is an IID sequence of random variables $X_{1},X_{2},\ldots$, then the distributions of $W_{r,k}^{(\rom{1})}$, $W_{r,k}^{(\rom{2})}$, $W_{r,k}^{(\rom{3})}$ and $W_{r,k}^{(\rom{4})}$ will be referred to as Type $\rom{1}$, $\rom{2}$, $\rom{3}$ and $\rom{4}$ negative binomial distributions of order $k$ and will be denoted as $NB_{k}^{(\rom{1})}(r,\theta)$, $NB_{k}^{(\rom{2})}(r,\theta)$, $NB_{k}^{(\rom{3})}(r,\theta)$, and $NB_{k}^{(\rom{4})}(r,\theta)$, respectively.\\
When the sequence is a $q$-geometric model, the distributions of $W_{r,k}^{(\rom{1})}$, $W_{r,k}^{(\rom{2})}$, $W_{r,k}^{(\rom{3})}$ and $W_{r,k}^{(\rom{4})}$ have been called Type $\rom{1}$, $\rom{2}$, $\rom{3}$ and $\rom{4}$ $q$-negative binomial distributions of order $k$, respectively. They can be denoted as $q-NB_{k}^{(\rom{1})}(r,\theta)$, $q-NB_{k}^{(\rom{2})}(r,\theta)$, $q-NB_{k}^{(\rom{3})}(r,\theta)$ and $q-NB_{k}^{(\rom{4})}(r,\theta)$, respectively.\\

According to the four aforementioned counting schemes, the random variables of the number of runs of length $k$ counted in $n$ outcomes have four distributions denoted as $N_{n,k}$, $G_{n,k}$, $M_{n,k}$, and $E_{n,k}$, respectively. Moreover, if the sequence is an IID sequence of random variables, $X_{1},X_{2},\ldots ,X_{n}$, then the distributions of $N_{n,k}$, $G_{n,k}$, $M_{n,k}$, and $E_{n,k}$ will be referred to as Type $\rom{1}$, $\rom{2}$, $\rom{3}$, and $\rom{4}$ binomial distributions of order $k$ and will be denoted as $B_{k}^{(\rom{1})}(n,\theta)$, $B_{k}^{(\rom{2})}(n,\theta)$, $B_{k}^{(\rom{3})}(n,\theta)$, and $B_{k}^{(\rom{4})}(n,\theta)$, respectively.\\
When the sequence is a $q$-geometric model, the distributions of $N_{n,k}$, $G_{n,k}$, $M_{n,k}$, and $E_{n,k}$ are called Type $\rom{1}$, $\rom{2}$, $\rom{3}$, and $\rom{4}$ $q$-binomial distribution of order $k$, respectively. These distributions can be denoted as $q-B_{k}^{(\rom{1})}(n,\theta)$, $q-B_{k}^{(\rom{2})}(n,\theta)$, $q-B_{k}^{(\rom{3})}(n,\theta)$, and $q-B_{k}^{(\rom{4})}(n,\theta)$, respectively.\\

To further clarify the distinctions among the aforementioned counting methods, we consider the example $n=12$; in this case, the binary sequence $011111000111$ contains $N_{12,2}=3$, $G_{12,2}=2$, $M_{12,2}=6$, $E_{12,5}=1$,  $W_{2,2}^{(\rom{1})}=5$, $W_{2,2}^{(\rom{2})}=11$, $W_{2,2}^{(\rom{3})}=4$, and $W_{2,3}^{(\rom{4})}>12$.

\citet{aki2000numbers} introduced a more generalized counting scheme called the $\ell$-overlapping counting scheme, where $\ell$ is a nonnegative integer less than $k$ (see also \citet{han2000unified}; \citet{antzoulakos2003waiting} ; \citet{inoue2003generalized}; \citet{makri2005binomial}; \citet{makri2007polya}; \citet{makri2015}). This scheme counts a success run of length  $k$ each of which may have an overlapping (common) part of length at most $\ell$ ($\ell=0,1,\ldots,k-1$) with the previous run of success of length $k$ that has already been enumerated. The nonoverlapping case ($\ell=0$) and the overlapping case ($\ell=k-1$) are special cases of
this scheme.

The random variable $W_{r,k,\ell}$ denotes the waiting time for the $r$-th occurrence of the $\ell$-overlapping success run of length $k$.
If the sequence is an IID sequence of random variables $W_{r,k,\ell}$, it will be referred to as a negative binomial distribution of order $k$ in the $\ell$-overlapping case and denoted as $NB_{k,\ell}(r,\theta)$.\\
According to the counting scheme mentioned earlier, the random variables of the number of $\ell$-overlapping success runs of length $k$ counted in $n$ outcomes are denoted as $N_{n,k,\ell}$. Moreover, if the sequence is an IID sequence of random variables, $X_{1},X_{2},\ldots ,X_{n}$, then distributions of $N_{n,k,\ell}$ will be referred to as binomial distributions of order $k$ in the $\ell$-overlapping case and denoted as $B_{k,\ell}(n,\theta)$.\\
When the sequence is a $q$-geometric model, the distributions of $W_{r,k,\ell}$ and $N_{n,k,\ell}$ are called  $q$-negative binomial distribution of order $k$ in the $\ell$-overlapping case and $q$-binomial distribution of order $k$ in the $\ell$-overlapping case, respectively. They are denoted as $q-NB_{k,\ell}(r,\theta)$ and $q-B_{n,k,\ell}(n,\theta)$, respectively.\\
To explain this better, let us assume that n = 15 binary trials, numbered from 1 to 15, are
performed, and we obtain the following outcomes 111111011110111. Then, the $\ell$-overlapping
1-runs of length 4 are as follows: 1,2,3,4; 3,4,5,6; 8,9,10,11 for $\ell = 2$, and 1,2,3,4; 2,3,4,5; 3,4,5,6;
8,9,10,11 for $\ell = 3$. Hence, $N_{15,4,2} = 3$ and $N_{15,4,3} = 4$.

Let $N_{n}^{(a)}$, $a=\rom{1},\rom{2},\rom{3}$ be a random variable denoting the number of occurrences of runs in the sequence of $n$ trials; $N_{n}^{(a)}$, $a=\rom{1},\rom{2},\rom{3}$, which is coincident with $N_{n,k}$, $G_{n,k}$, and $M_{n,k}$, respectively. The random variable $N_{n}$ is closely related to the random variable $W_{r,k}$ (see \citet{feller1968introduction}). We have the following dual relationship $N_{n}^{(a)}<r$ if and only if $W_{r,k}^{(a)}>n$ for $a=\rom{1},\rom{2},\rom{3}$. Now, the probability function of the $q$-binomial distribution of order $k$, can be derived easily using the dual relationship between the binomial and negative binomial distributions of order $k$:
\begin{equation*}
\begin{split}
&P_{q,\theta}\left(N_{n}^{(\rom{1})}<r\right)=P_{q,\theta}\left(W_{r,k}^{(\rom{1})}>n\right),\ P_{q,\theta}\left(N_{n}^{(\rom{2})}<r\right)=P_{q,\theta}\left(W_{r,k}^{(\rom{2})}>n\right)\\
&P_{q,\theta}\left(N_{n}^{(\rom{3})}<r\right)=P_{q,\theta}\left(W_{r,k}^{(\rom{3})}>n\right),\ P_{q,\theta}\left(N_{n,\ell}<r\right)=P_{q,\theta}\left(W_{r,k,\ell}>n\right).
\end{split}
\end{equation*}
However, the above mentioned dual relationship cannot be considered by the Type $\rom{4}$ enumeration scheme. Because  $W_{r,k}^{(\rom{4})}>n$ implies $N_{n}^{(\rom{4})}<r$, instead of an iff relationship.

The $q$-negative binomial distribution with parameters $q$, $\theta$, $k$, and $r$ is a distribution of the length of a sequence of Bernoulli trials with geometrically increasing failure probability until the $r$-th appearance of a success run of length $k$. Here, $0 < q, \theta < 1$, and $r$ and $k$ are positive integers. In this study, we examined the waiting time distribution for the $r$-th appearance of a success run of length $k$ (nonoverlapping, at least, overlapping, and $\ell$-overlapping), and the probability of ones vary according to a geometric rule. The remainder of this paper is organized as follows: In Section 2, we introduce the basic definitions and necessary notations that will be useful throughout this article. In sections 3, 4, 5, and 6, we examine the Type \rom{1} $q$-negative, Type \rom{2} $q$-negative, Type \rom{3} $q$-negative and Type \rom{4} $q$-negative binomial distributions of order $k$, respectively, and explain the derivation of their exact probability functions via combinatorial analysis. In Section 7, we examine the waiting time for the $r$-th $\ell$-overlapping occurrence of a success run of length $k$ in a $q$-geometric sequence. The exact probability function of the $q$-negative binomial distribution of order $k$ in the $\ell$-overlapping case was derived via combinatorial analysis.

\section{Terminology and notation}
We recall some definitions, notations, and known results that will be used in this paper. Throughout this work, we suppose that $0<q<1$. First, we introduce the following notation.
\begin{itemize}
\item $S_{n}:$ the total number of successes in $X_{1}, X_{2}, \ldots, X_{n}$;
\item $F_{n}:$ the total number of failures in $X_{1}, X_{2}, \ldots, X_{n}$.
\end{itemize}
\noindent
Next, we introduce some basic $q$-sequences and functions and their properties, which are useful in the remainder of the paper. The $q$-shifted factorials are defined as
\begin{equation}\label{}
(a;q)_{0}=1,\quad (a;q)_{n} =\prod_{k=0}^{n-1}(1-aq^{k}),\quad (a;q)_{\infty} =\prod_{k=0}^{\infty}(1-aq^{k}).
\end{equation}
\noindent
Let $m$, $n$, and $i$ be positive integers and $z$ and $q$ be real numbers, with $q\neq 1$. The number $[z]_{q} = (1-q^{z})/(1-q)$ is called the $q$-number, and in particular, $[z]_{q}$ is called the $q$-integer. The $m$ th order factorial of the $q$-number $[z]_{q}$, which is defined by
\begin{equation}\label{}
\begin{split}
[z]_{m,q}&=\prod_{i=1}^{m}[z-i+1]_{q}= [z]_{q} [z-1]_{q}\cdots[z-m+1]_{q}\\
&=\frac{(1-q^{z})(1-q^{z-1})\cdots(1-q^{z-m+1})}{(1-q)^{m}},\ z=1,2,\ldots,\ m=0,1,\ldots,z.
\end{split}
\end{equation}
\noindent
Is called $q$-factorial of $z$ of order $m$. In particular, $[m]_{q}! = [1]_{q}[2]_{q}...[m]_{q}$ is called $q$-factorial of $m$. The $q$-binomial coefficient (or Gaussian polynomial) is defined as follows:
\noindent
\begin{equation}\label{}
\begin{split}
\begin{bmatrix}
n\\
m
\end{bmatrix}_{q}
&=\frac{[n]_{m,q}}{[m]_{q}!}=\frac{[n]_{q}!}{[m]_{q}![n-m]_{q}!}=\frac{(1-q^{n})(1-q^{n-1})\cdots(1-q^{n-m+1})}{(1-q^{m})(1-q^{m-1})\cdots(1-q)}\\
&=\frac{(q;q)_{n}}{(q;q)_{m}(q;q)_{n-m}},\ m=1,2,\ldots,
\end{split}
\end{equation}
\noindent
The $q$-binomial ($q$-Newton's binomial) formula is expressed as follows:
\noindent
\begin{equation}\label{}
\prod_{i=1}^{n}(1+zq^{i-1})=\sum_{k=0}^{n}q^{k(k-1)/2}\begin{bmatrix}
n\\
k
\end{bmatrix}_{q}
z^{k},\ -\infty<z<\infty,\ n=1,2,\ldots.
\end{equation}
\noindent
For $q\rightarrow 1$, the $q$-analogs correspond to their classical counterparts, that is,
\noindent
\begin{equation*}\label{eq: 1.1}
\begin{split}
\lim\limits_{q\rightarrow 1}
\begin{bmatrix}
n\\
r
\end{bmatrix}_{q}
={n \choose r}
\end{split}
\end{equation*}
\noindent
Now, let us consider the sequence of independent geometric sequences of trials with the probability of failure at the $i$th geometric sequence of trials given by \eqref{failureprobofq} or \eqref{failureprobofq2}. We now focus on studying the number of successes in a given number of trials in this stochastic model.
\begin{definition}
Let $Z_{n}$ be the number of successes in a sequence of $n$ independent Bernoulli trials, with the probability of success at the $i$-th geometric sequence of trials given by \eqref{failureprobofq} or \eqref{failureprobofq2}. The distribution of the random variable $Z_{n}$ is called the $q$-binomial distribution, with parameters $n$, $\theta$, and $q$.
\end{definition}
\noindent
Let us introduce a $q$-analog of the binomial distribution. The probability function of the number $Z_{n}$ of successes in $n$ trials $X_{1}, \ldots, X_{n}$ is given by
\begin{equation}\label{q-binomialdist}
\begin{split}
P_{q,\theta}\{Z_{n}=r\}=
\begin{bmatrix}
n\\
r
\end{bmatrix}_{q}
\theta^{r}\prod_{i=1}^{n-r}(1-\theta q^{i-1}),
\end{split}
\end{equation}\\
For $r = 0, 1, . . . , n$, $0 < q < 1$. This distribution  is called a $q$-binomial distribution. For $q \rightarrow 1$, because
\begin{equation*}\label{eq: 1.1}
\begin{split}
\lim\limits_{q\rightarrow 1}
\begin{bmatrix}
n\\
r
\end{bmatrix}_{q}
={n \choose r}
\end{split}
\end{equation*}
the $q$-binomial distribution converges to the usual binomial distribution as $q \rightarrow 1$, as follows:
\begin{equation}
P_{\theta}\left(Z_{n}=r\right)={n\choose r}\theta^{r}(1-\theta)^{n-r},\ r=0,1,\ldots,n,
\end{equation}
with parameters $n$ and $\theta$.
The $q$-binomial distribution was studied by \citet{charalambides2010a,Charalambides2016} and is related t6 the $q$-Berstein polynomial. \citet{Jing1994} introduced probability function \eqref{q-binomialdist} as a $q$-deformed binomial distribution, and derived the recurrence relation of its probability distribution.
\noindent
In the sequel, $P_{q ,\theta}(.)$ and $P_{\theta}(.)$ denote the probabilities related to the stochastic models \eqref{failureprobofq} and \eqref{bernoullifailureprob}, respectively.

\section{Type \rom{1} $q$-negative binomial distribution of order $k$}
In this section, we will study the Type $\rom{1}$ $q$-negative binomial distribution of order $k$. Let us consider the waiting time for the $r$-th occurrence of a success run of length $k$. For $r\in N$ and $k\in N$, let $W_{r,k}^{(\rom{1})}$ be the waiting time for the $r$-th appearance of a run of successes of length $k$. We use the nonoverlapping counting scheme (Type \rom{1} enumeration scheme;  \citet{feller1968introduction}). In this scheme, once $k$ consecutive successes show up in a run of successes of length $k$ are counted, the counting procedure starts anew (from scratch). The support (range set) of $W_{r,k}^{(\rom{1})}$, $\mathfrak{R}\left(W_{r,k}^{(\rom{1})}\right)$ is given by
\noindent
\begin{equation*}
\begin{split}
\mathfrak{R}\left(W_{r,k}^{(\rom{1})}\right)=\{kr,kr+1,\ldots\}.
\end{split}
\end{equation*}
\noindent
We now state a useful definition and lemma for the proofs of the theorem in the sequel.
\begin{definition}
For $0<q\leq1$, define
\begin{equation*}\label{eq: 1.1}
\begin{split}
 A_{q}^{k}(r,s,t)=\sum_{\substack{y_{1},y_{2},\ldots,y_{r}}}   q^{y_{2}+2y_{3}+\cdots+(r-1) y_{r}}\\
\end{split}.
\end{equation*}

\noindent
where the summation is performed over all integers $y_1,\ldots,y_{r}$ satisfying
\noindent

\begin{equation*}\label{eq:1}
\begin{split}
y_{1}+y_{2}+\cdots+y_{r}=s,
\end{split}
\end{equation*}
\noindent
\begin{equation*}\label{eq:1}
\begin{split}
\left[\frac{y_{1}}{k}\right]+\left[\frac{y_{2}}{k}\right]+\cdots+\left[\frac{y_{r}}{k}\right]=t,\ \text{and}
\end{split}
\end{equation*}
\noindent
\begin{equation*}\label{eq:2}
\begin{split}
y_{j}\geq 0,\quad j=1,\ \ldots,\ r.
\end{split}
\end{equation*}

\end{definition}
\noindent
The following gives a recurrence relation useful for computing of $A_{q}^{k}(r,s,t)$.\\
\noindent

\begin{lemma}
\label{aq_fun}
{\rm{[\citet{yalcin2014q}]}} For $0<q\leq1$, $A_{q}^{k}(r,s,t)$ obeys the following recurrence relation:
\noindent
\begin{equation*}\label{eq: 1.1}
\begin{split}
A_{q}^{k}&(r,s,t)\\
=&\left\{
  \begin{array}{ll}
    \sum_{j=0}^{k-1} q^{(r-1)j} A_{q}^{k}(r-1,s-j,t)\\
    +\sum_{j=k}^{s} q^{(r-1)j}A_{q}^{k}\left(r-1,s-j,t-\left[\frac{j}{k}\right]\right), & \text{if $r>1$, $s\geq0$ and $t\geq 0$} \\
    1, & \text{if $r=1$, $s\geq0$ and $t=\left[\frac{s}{k}\right]$} \\
    0, & \text{otherwise.}\\
  \end{array}
\right.
\end{split}
\end{equation*}

\end{lemma}

\begin{remark}
{\rm
We observe that $A_{1}^{k}(r,s,t)$ is the number of integer solutions $(y_{1},\ \ldots,\ y_{r})$ of
\noindent
\begin{equation*}\label{eq:1}
\begin{split}
y_{1}+y_{2}+\cdots+y_{r}=s,
\end{split}
\end{equation*}
\noindent
\begin{equation*}\label{eq:1}
\begin{split}
\left[\frac{y_{1}}{k}\right]+\left[\frac{y_{2}}{k}\right]+\cdots+\left[\frac{y_{r}}{k}\right]=t,\ \text{and}
\end{split}
\end{equation*}
\noindent
\begin{equation*}\label{eq:2}
\begin{split}
y_{j}\geq 0,\quad j=1,\ \ldots,\ r,
\end{split}
\end{equation*}
\noindent
indicating the total number of arrangements of $s$ balls in $r$ distinguishable cells. Hence, $t$ nonoverlapping runs of balls of length $k$ is given by
\noindent
\begin{equation*}\label{eq: 1.1}
\begin{split}
A_{1}^{k}(r,s,t)={r+t-1 \choose t}S(s-kt,r,k-1),
\end{split}
\end{equation*}
\noindent
where $S(a,\ b,\ c)$ denotes the total number of integer solutions $x_{1}+x_{2}+\cdots+x_{a}=c$ such that $0<x_{i}<b$ for $i=1,2,\ldots,a$. The number can be expressed as
\noindent
\begin{equation*}\label{eq: 1.1}
\begin{split}
S(a,\ b,\ c)=\sum_{j=0}^{a}(-1)^{j}{a \choose j}{c-j(b-1)-1 \choose a-1}.
\end{split}
\end{equation*}
\noindent
For example, see \citet{charalambides2002enumerative}.
}
\end{remark}
\noindent
The probability function of the Type $\rom{1}$ $q$-negative binomial distribution of order $k$ is obtained from the following theorem. It is evident that
\noindent
\begin{equation*}
P_{q,\theta}\left(W_{r,k}^{(\rom{1})}=n\right)=0\ \text{for}\ 0\leq n<rk
\end{equation*}
\noindent
and hence, we focus on determining the probability mass function (PMF) for $n\geq rk$.
\noindent
\begin{theorem}
\label{pmftype1q-negative}
The PMF $w_{q}^{(\rom{1})}(n;r,k;\theta)=P_{q,\theta}\left(W_{r,k}^{(\rom{1})}=n\right)$ is given by
\begin{equation*}\label{pmftype1q-nega}
\begin{split}
&w_{q}^{(\rom{1})}(n;r,k;\theta)=\\
&\left\{
  \begin{array}{ll}
    {\sum_{i=1}^{n-rk}} \theta^{n-i}q^{ik}{\prod_{{j}=1}^i}(1-\theta q^{j-1})  A_{q}^{k}(i,n-k-i,r-1)  & \text{if $n\geq rk+1$,} \\
    \theta^{kr} & \text{if $n=rk$,} \\
    0, & \text{if n $<$ rk.}\\
  \end{array}
\right.
\end{split}
\end{equation*}
\end{theorem}

\begin{proof}
First, we examine $w_{q}^{(\rom{1})}(rk;r,k;\theta)$. Clearly, $w_{q}^{(\rom{1})}(rk;r,k;\theta)=\big(\theta q^{0}\big)^{k}=\theta^{rk}$. Hereinafter, we assume $n > rk.$ By the definition of $W_{r,k}^{(\rom{1})}$ every sequence of $n$ binary trials belonging to event $W_{r,k}^{(\rom{1})}=n$ must end with $k$ successes. The event $W_{r,k}^{(\rom{1})}=n$ can be expressed as
$$\left\{W_{r,k}^{(\rom{1})}=n\right\}=\left\{N_{n-k,k}=r-1\ \wedge\ X_{n-k+1}=\cdots=X_{n}=1\right\}.$$
\noindent
We partition the event $W_{r,k}^{(\rom{1})}=n$ into disjointed events given by $F_{n}=i,$ for $i=1,\ldots,n-rk.$ Adding the probabilities, we have
\noindent
\begin{equation*}\label{eq:bn}
\begin{split}
P_{q,\theta}\left(W_{r,k}^{(\rom{1})}=n\right)=\sum_{i=1}^{n-rk}P_{q,\theta}\Big(N_{n-k,k}=&r-1\ \wedge\ F_{n}=i\ \wedge\\
&X_{n-k+1}=\cdots=X_{n}=1\Big).\\
\end{split}
\end{equation*}
\noindent
If the number of zeroes in the first $n-k$ trials is equal to $i,$ that is, $F_{n-k}=i,$ then for each of the $(n-k+1)$ to $n$-th trials, the probability of success is
\noindent
\begin{equation*}\label{eq: kk}
\begin{split}
p_{n-k+1}=\cdots=p_{n}=\theta q^{i}.
\end{split}
\end{equation*}
\noindent
This probability can be rewritten as follows.
\noindent
\begin{equation*}\label{eq:bn1}
\begin{split}
P_{q,\theta}\left(W_{r,k}^{(\rom{1})}=n\right)=\sum_{i=1}^{n-rk}&P_{q,\theta}\Big(N_{n-k,k}=r-1\ \wedge\ F_{n-k}=i\Big)\\
&\times P_{q,\theta}\Big(X_{n-k+1}=\cdots =X_{n}=1\mid F_{n-k}=i\Big)\\
=\sum_{i=1}^{n-rk}&P_{q,\theta}\Big(N_{n-k,k}=r-1\ \wedge\ S_{n}=i\Big)\Big(\theta q^{i}\Big)^{k}.\\
\end{split}
\end{equation*}
\noindent
An element of the event $\left\{W_{r,k}^{(\rom{1})}=n,\ F_{n}=i\right\}$
is an ordered sequence consisting of $n-i$ successes and $i$ failures such that the length of the success run is nonnegative integer, $r$ nonoverlapping runs of success of length $k$ and end with $k$ successes. The number of these sequences can be derived as follows: First, we distribute the $i$ failures. Note that $i$ failures form $i+1$ cells. Next, we will distribute the $n-i-k$ successes into $i+1$ distinguishable cells as follows.
\noindent
\begin{equation*}\label{eq: 1.1}
\begin{split}
\overbrace{{\underbrace{1\ldots1}_{y_{1}}}0{\underbrace{1\ldots1}_{y_{2}}}0\ldots0{\underbrace{1\ldots1}_{y_{i}}}0{\underbrace{1\ldots1}_{y_{i+1}}}}^{n-k}{\boxed{{{\underbrace{1\ldots1}_{k}}}}}
\end{split}
\end{equation*}
\noindent
with $i$ 0s and $n-i$ 1s, where the length of the first 1-run is $y_{1}$, the length of the second 1-run is $y_{2}$,..., the length of the $(i+1)$-th $1$-run is $y_{i+1}$. The probability of the event $\left\{W_{r,k}^{(\rom{1})}=n,\ F_{n}=i\right\}$ is given by
\noindent
\begin{equation*}\label{eq: 1.1}
\begin{split}
&(\theta q^{0})^{y_{1}}(1-\theta q^{0})(\theta q^{1})^{y_{2}}(1-\theta q^{1})\cdots(\theta q^{i-1})^{y_{i}}(1-\theta q^{i-1})(\theta q^{i})^{y_{i+1}+k}.
\end{split}
\end{equation*}
\noindent
Using simple exponentiation algebra arguments to simplify,
\noindent
\begin{equation*}\label{eq: 1.1}
\begin{split}
\theta^{n-i}q^{ik}{\prod_{{j}=1}^i}\left(1-\theta q^{j-1}\right)q^{y_{2}+2y_{3}+\cdots+i y_{i+1}}.
\end{split}
\end{equation*}
\noindent
However, $y_{j}$s are nonnegative integers such that ${y_{1}}+{y_{2}}+\cdots+{y_{i+1}}=n-k-i$ and
\noindent
\begin{equation*}\label{eq: 1.1}
\begin{split}
\left[\frac{y_{1}}{k}\right]+\left[\frac{y_{2}}{k}\right]+\cdots+\left[\frac{y_{i+1}}{k}\right]=r-1\\
\end{split}
\end{equation*}
\noindent
so that
\noindent
\begin{equation*}\label{eq: 1.1}
\begin{split}
&P_{q,\theta}\left(W_{r,k}^{(\rom{1})}=n,\ F_{n}=i\right)\\
&=\theta^{n-i}q^{ik}{\prod_{{j}=1}^i}\left(1-\theta q^{j-1}\right)  {\mathop{\sum...\sum}_{\substack{y_{1}+y_{2}+\cdots+y_{i+1}=n-k-i\\\left[\frac{y_{1}}{k}\right]+\left[\frac{y_{2}}{k}\right]+\cdots+\left[\frac{y_{i+1}}{k}\right]=r-1\\\\y_{1}\geq 0,\ldots,y_{i+1}\geq 0}}}   q^{y_{2}+2y_{3}+\cdots+i y_{i+1}}.\\
\end{split}
\end{equation*}
\noindent
Summing with respect to $i=1, ..., n-rk$, we get\\
\noindent
\begin{equation*}\label{eq: 1.1}
\begin{split}
{\sum_{i=1}^{n-rk}} \theta^{n-i}q^{ik}{\prod_{{j}=1}^i}\left(1-\theta q^{j-1}\right)  {\mathop{\sum...\sum}_{\substack{y_{1}+y_{2}+\cdots+y_{i+1}=n-k-i\\\left[\frac{y_{1}}{k}\right]+\left[\frac{y_{2}}{k}\right]+\cdots+\left[\frac{y_{i+1}}{k}\right]=r-1\\\\y_{1}\geq 0,\ldots,y_{i+1}\geq 0}}}   q^{y_{2}+2y_{3}+\cdots+i y_{i+1}}\\
\end{split}
\end{equation*}
\noindent
From lemma \ref{aq_fun}, we can rewrite the above as follows:
\noindent
\begin{equation*}\label{eq: 1.1}
\begin{split}
{\sum_{i=1}^{n-rk}} \theta^{n-i}q^{ik}{\prod_{{j}=1}^i}\left(1-\theta q^{j-1}\right)  A_{q}^{k}(i,n-k-i,r-1).
\end{split}
\end{equation*}
\noindent
Thus, the proof is completed.
\end{proof}
\noindent
For $q=1$, from Theorem \ref{pmftype1q-negative}, the PMF of the Type \rom{1} negative binomial distribution of order $k$ in Bernoulli trials with success probability $\theta$ is obtained as follows:
\noindent
\begin{corollary}
\label{cor:type1negativ}
The PMF $w^{(\rom{1})}(n;r,k;\theta)=P_{\theta}\left(W_{r,k}^{(\rom{1})}=n\right)$ is given by
\begin{equation}
\label{type1negativ}
\begin{split}
P_{\theta}\left(W_{r,k}^{(\rom{1})}=n\right)=
\left\{
  \begin{array}{ll}
    {\sum_{i=1}^{n-rk}} \theta^{n-i}\left(1-\theta \right)^{i}  A_{1}^{k}(i,n-k-i,r-1)  & \text{if}\ n\geq rk+1, \\
    \theta^{kr} & \text{if $n=rk$,} \\
    0, & \text{if n $<$ rk.}\\
  \end{array}
\right.
\end{split}
\end{equation}
\end{corollary}
\noindent
Notably, Equation \eqref{type1negativ} in Corollary \ref{cor:type1negativ} is an alternative to the formula proposed by \cite{philippou1984negative} (See Definition 2.1) and computationally simpler.
\noindent

\begin{remark}
{\rm
A random variable related to $W_{r,k}^{(\rom{1})}$ is $N_{n,k}$ denotes the number of occurrences of a success run of length $k$ in the sequence of $n$ trials. Because the events $\left(N_{n,k}\geq r\right)$ and $\left(W_{r,k}^{(\rom{1})}\leq n\right)$ are equivalent, an alternative formula for the PMF of Type $\rom{1}$ $q$-negative binomial distribution of order $k$, can be easily obtained, using the dual relationship between the binomial and negative binomial distribution of order $k$. This alternative formula is given by
\noindent
\begin{equation*}
P_{q,\theta}\left(N_{n,k}\geq r\right)=P_{q,\theta}\left(W_{r,k}^{(\rom{1})}\leq n\right).
\end{equation*}
\noindent
Consequently, the PMF $w_{q}^{(\rom{1})}(n;k,r;\theta)=P_{q,\theta}(W_{r,k}^{(\rom{1})}=n)$ is implicitly determined by
\begin{equation}
\label{alpmf:type1q-nega}
w_{q}^{(\rom{1})}(n;k,r;\theta)=\sum_{x=0}^{r-1}f_{q}^{(\rom{1})}(x;n-1,k;\theta)-f_{q}^{(\rom{1})}(x;n,k;\theta)\ \text{for}\ n\geq rk\ \text{and}\  r\geq 1,
\end{equation}
where the probabilities $f_{q}^{(\rom{1})}(x;n-1,k;\theta)=P_{q,\theta}(N_{n-1,k}=x)$ and $f_{q}^{(\rom{1})}(x;n,k;\theta)=P_{q,\theta}(N_{n,k}=x)$ were obtained by \citet[Theorem~2]{yalcin2014q}:
\begin{equation}
\label{pmf:q-binomial}
f_{q}^{(\rom{1})}(x;n,k;\theta)={\sum_{i=0}^{n-xk}} \theta^{n-i}{\prod_{j=1}^i}(1-\theta q^{j-1})A_{q}^{k}(i+1,n-i,x)\ \text{for}\  x=0,1,\ldots,\left[\frac{n}{k}\right].
\end{equation}
Usually, the obtained expression \eqref{pmf:q-binomial} for $P_{q,\theta}\left(W_{r,k}^{(\rom{1})}=n\right)$ is computationally faster than that obtained using \eqref{pmftype1q-nega}.
}
\end{remark}



\section{Type \rom{2} $q$-negative binomial distribution of order $k$}
\vspace{0.13in}
In this section, we examine the Type $\rom{2}$ $q$-negative binomial distribution of order $k$. Let us consider the waiting time for the $r$-th occurrence of a success run of length at least $k$. For $r\in N$ and $k\in N$, let $W_{r,k}^{(\rom{2})}$be the waiting time for the $r$-th appearance of a run of successes of length at least $k$. We use the at least counting scheme (Type \rom{2} enumeration scheme, by \citet{mood1940distribution}), i.e., a run of successes of length greater than or equal to $k$ preceded and succeeded by failure or by nothing. The support (range set) of $W_{r,k}^{(\rom{2})}$, $\mathfrak{R}\left(W_{r,k}^{(\rom{2})}\right)$ is given by
\noindent
\begin{equation*}
\begin{split}
\mathfrak{R}\left(W_{r,k}^{(\rom{2})}\right)=\{r(k+1)-1,r(k+1),\ldots\}.
\end{split}
\end{equation*}
\noindent
We now make a useful definition and lemma for the proofs of the Theorem in the sequel.
\noindent
\begin{definition}
For $0<q\leq1$, we define
\begin{equation*}\label{eq: 1.1}
\begin{split}
 B_{q}^{k}(r,s,t)=\sum_{\substack{y_{1},y_{2},\ldots,y_{r}}}   q^{y_{2}+2y_{3}+\cdots+(r-1) y_{r}}\\
\end{split}.
\end{equation*}
\noindent
where the summation is over all integers $y_1,\ldots,y_{r}$ satisfying
\noindent
\begin{equation*}\label{eq:1}
\begin{split}
y_{1}+y_{2}+\cdots+y_{r}=s,
\end{split}
\end{equation*}
\noindent
\begin{equation*}\label{eq:1}
\begin{split}
\sum_{i=1}^{r}I(y_{i}-k)=t,\ \text{and}
\end{split}
\end{equation*}
\noindent
\begin{equation*}\label{eq:2}
\begin{split}
y_{j}\geq 0,\quad j=1,\ldots,r,
\end{split}
\end{equation*}
where $I(j)$ stands for the indicator function such that $I(j)=1$, if $j\geq 0$; $0$, otherwise.
\end{definition}
\noindent
The following gives a recurrence relation useful for the computation of $B_{q}^{k}(r,s,t)$.\\
\noindent
\begin{lemma}
\label{bq_fun}
{\rm[\citet{makri2016runs}]} For $0<q\leq1$, $B_{q}^{k}(r,s,t)$ obeys the following recurrence relation,
\noindent
\begin{equation*}\label{eq: 1.1}
\begin{split}
B_{q}^{k}&(r,s,t)\\
=&\left\{
  \begin{array}{ll}
    \sum_{j=0}^{s-(t-1)k} q^{(r-1)j} B_{q}^{k}(r-1,s-j,t-I(j-k)), & \text{if $r>1$, $s\geq tk$ and $t\leq r$} \\
    1, & \text{if $r=1$, $s\geq k$ and $t=1$} \\
    0, & \text{otherwise.}\\
  \end{array}
\right.
\end{split}
\end{equation*}
\end{lemma}

\begin{remark}
{\rm
We observe that $B_1^{k}(r,s,t)$ is the number of integer solutions $(y_{1},\ldots,y_{r})$ of
\noindent

\begin{equation*}\label{eq:1}
\begin{split}
y_{1}+y_{2}+\cdots+y_{r}=s,
\end{split}
\end{equation*}
\noindent
\begin{equation*}\label{eq:1}
\begin{split}
\sum_{i=1}^{r}I(y_{i}-k)=t,\ \text{and}
\end{split}
\end{equation*}
\noindent
\begin{equation*}\label{eq:2}
\begin{split}
y_{j}\geq 0,\quad j=1,\ \ldots,\ r.
\end{split}
\end{equation*}
\noindent
This indicates the total number of arrangements of $s$ balls in $r$ distinguishable cells such that each of exactly $t$ of them receives at least $k$ balls. This number is given by
\noindent
\begin{equation*}\label{eq: 1.1}
\begin{split}
B_{1}^{k}(r,s,t)={r \choose t}H_{r-t}(s-tk,r,k-1).
\end{split}
\end{equation*}
\noindent
where $H_{m}(\alpha,r,k)$ denotes the number of allocations of $\alpha$ indistinguishable balls into $r$ distinguishable cells, where each of the $m$ $(0\leq m\leq r)$ specified cells is occupied by at most $k$ balls. The number can be expressed as
\noindent
\begin{equation*}\label{eq: 1.1}
\begin{split}
H_{m}(\alpha,r,k)=\sum_{j=0}^{\left[\frac{\alpha}{k+1}\right]}(-1)^{j}{m \choose j}{\alpha-(k+1)j+r-1 \choose \alpha-(k+1)j}.
\end{split}
\end{equation*}
\noindent
(\citet{makri2007success}).
}
\end{remark}

\noindent
The probability function of the Type \rom{2} $q$-negative binomial distribution of order $k$ is obtained from Theorem \ref{thmtype2q-nega}: Clearly,
\noindent
\begin{equation*}
P_{q,\theta}\left(W_{r,k}^{(\rom{2})}=n\right)=0\ \text{for}\ 0\leq n<r(k+1)-1
\end{equation*}
\noindent
and so we focus on determining the PMF for $n\geq r(k+1)-1$.
\noindent
\begin{theorem}
\label{thmtype2q-nega}
The PMF $w_{q}^{(\rom{2})}\left(n;r,k;\theta\right)=P_{q,\theta}\left(W_{r,k}^{(\rom{2})}=n\right)$ is given by\begin{equation*}\label{eq: 1.1}
\begin{split}
&P_{q,\theta}\left(W_{r,k}^{(\rom{2})}=n\right)=\\
&\left\{
  \begin{array}{ll}
    {\sum_{i=r-1}^{n-rk}} \theta^{n-i}q^{ik}{\prod_{{j}=1}^i}(1-\theta q^{j-1})  B_{q}^{k}(i,n-k-i,r-1)  & \text{if $n> r(k+1)-1$,} \\
    \theta^{r(k+1)-1} & \text{if $n= r(k+1)-1$,} \\
    0, & \text{if n $<$  $r(k+1)-1$.}\\
  \end{array}
\right.
\end{split}
\end{equation*}
\end{theorem}

\begin{proof}
We first consider $w_{q}^{(\rom{2})}(r(k+1)-1;r,k;\theta)$. Clearly, $w_{q}^{(\rom{2})}(r(k+1)-1;r,k;\theta)=\big(\theta q^{0}\big)^{ r(k+1)-1}=\theta^{ r(k+1)-1}$. Hereinafter, $n > r(k+1)-1.$ From the definition of $W_{r,k}^{(\rom{2})}$ every sequence of $n$ binary trials belonging to event $W_{r,k}^{(\rom{2})}=n$ must end with $k$ successes preceded by a failure. The event $W_{r,k}^{(\rom{2})}=n$ can be expressed as$$\left\{W_{r,k}^{(\rom{2})}=n\right\}=\left\{G_{n-k-1,k}=r-1\ \wedge\ X_{n-k}=0\ \wedge\ X_{n-k+1}=\cdots=X_{n}=1\right\}.$$
\noindent
We partitioned the event $W_{r,k}^{(\rom{2})}=n$ into disjointed events given by $F_{n}=i,$ for $i=r-1,\ldots,n-rk.$ Adding the probabilities we have
\noindent
\begin{equation*}\label{eq:bn}
\begin{split}
P_{q,\theta}\left(W_{r,k}^{(\rom{2})}=n\right)=\sum_{i=r-1}^{n-rk}P_{q,\theta}\Big(G_{n-k-1,k}=r-1\ \wedge\ &X_{n-k}=0\ \wedge\  F_{n}=i\ \wedge\\
&X_{n-k+1}=\cdots=X_{n}=1\Big).\\
\end{split}
\end{equation*}
\noindent
If the number of zeroes in the first $n-k-1$ trials is equal to $i-1,$ that is, $F_{n-k-1}=i-1,$ then in each of the $(n-k+1)$ to $n$-th trials the probability of success is
\noindent
\begin{equation*}\label{eq: kk}
\begin{split}
p_{n-k+1}=\cdots=p_{n}=\theta q^{i}.
\end{split}
\end{equation*}
\noindent
This probability can be written as follows.
\noindent
\begin{equation*}\label{eq:bn1}
\begin{split}
P_{q,\theta}&\left(W_{r,k}^{(\rom{2})}=n\right)\\
=&\sum_{i=r-1}^{n-rk}P_{q,\theta}\Big(G_{n-k-1,k}=r-1\ \wedge\ F_{n-k-1}=i-1\Big)\\
&\quad\quad\quad \times P_{q,\theta}\Big(X_{n-k}=0\ \wedge\ X_{n-k+1}=\cdots=X_{n}=1\mid F_{n-k-1}=i-1\Big)\\
=&\sum_{i=r-1}^{n-rk}P_{q,\theta}\Big(G_{n-k-1,k}=r-1\ \wedge\ F_{n-k-1}=i-1\Big)\Big(1-\theta q^{i-1}\Big)\Big(\theta q^{i}\Big)^{k}.\\
\end{split}
\end{equation*}
\noindent
An element of the event $\left\{W_{r,k}^{(\rom{2})}=n,\ F_{n}=i\right\}$ is an ordered sequence that consists of $n-i$ successes and $i$ failures such that the length of the success run is a nonnegative integer, and $r$ nonoverlapping runs of success of length at least $k$ end with $k$ successes preceded by a failure. The number of these sequences can be derived as follows: First, we distribute the $i$ failures. Note that $i$ failures form $i+1$ cells. Next, we distribute the $n-i-k$ successes into $i$ distinguishable cells as follows.
\begin{equation*}\label{eq: 1.1}
\begin{split}
\overbrace{{\underbrace{1\ldots1}_{y_{1}}}0{\underbrace{1\ldots1}_{y_{2}}}0\ldots0{\underbrace{1\ldots1}_{y_{i-1}}}0{\underbrace{1\ldots1}_{y_{i}}}}^{n-k-1}\overbrace{{\boxed{{0{\underbrace{1\ldots1}_{k}}}}}}^{k+1}
\end{split}
\end{equation*}
with $i$ 0s and $n-i$ 1s, where the length of the first 1-run is $y_{1}$, the length of the second 1-run is $y_{2}$,...; the length of the $(i)$-th $1$-run is $y_{i}$. The probability of the event $\left\{W_{r,k}^{(\rom{2})}=n,\ F_{n}=i\right\}$ is given by
\noindent
\begin{equation*}\label{eq: 1.1}
\begin{split}
(\theta q^{0})^{y_{1}}(1-\theta q^{0})(\theta q^{1})^{y_{2}}(1-\theta q^{1})\cdots(\theta q^{i-1})^{y_{i}}(1-\theta q^{i-1})(\theta q^{i})^{k}.
\end{split}
\end{equation*}
\noindent
Using simple exponentiation algebra arguments to simplify, we get
\noindent
\begin{equation*}\label{eq: 1.1}
\begin{split}
\theta^{n-i}q^{ik}{\prod_{{j}=1}^i}\left(1-\theta q^{j-1}\right)q^{y_{2}+2y_{3}+\cdots+(i-1) y_{i}}.
\end{split}
\end{equation*}
\noindent
However, the $y_{j}$s are nonnegative integers such that ${y_{1}}+{y_{2}}+\cdots+{y_{i}}=n-k-i$ and
\begin{equation*}\label{eq: 1.1}
\begin{split}
I(y_1-k)+\cdots+I(y_i-k)=r-1
\end{split}
\end{equation*}
so that
\noindent
\begin{equation*}\label{eq: 1.1}
\begin{split}
&P_{q,\theta}\left(W_{r,k}^{(\rom{2})}=n,\ F_{n}=i\right)\\
&=\theta^{n-i}q^{ik}{\prod_{{j}=1}^i}\left(1-\theta q^{j-1}\right)  {\mathop{\sum...\sum}_{\substack{y_{1}+y_{2}+\cdots+y_{i}=n-k-i\\I(y_1-k)+\cdots+I(y_i-k)=r-1\\y_{1}\geq 0,\ldots,y_{i}\geq 0}}}   q^{y_{2}+2y_{3}+\cdots+(i-1) y_{i}}.\\
\end{split}
\end{equation*}
\noindent
Summing the above with respect to $i=r-1, ..., n-rk$ yields\\
\noindent
\begin{equation*}\label{eq: 1.1}
\begin{split}
{\sum_{i=r-1}^{n-rk}} \theta^{n-i}q^{ik}{\prod_{{j}=1}^i}\left(1-\theta q^{j-1}\right)  {\mathop{\sum...\sum}_{\substack{y_{1}+y_{2}+\cdots+y_{i}=n-k-i\\I(y_1-k)+\cdots+I(y_j-k)=r-1\\y_{1}\geq 0,\ldots,y_{i}\geq 0}}}   q^{y_{2}+2y_{3}+\cdots+(i-1) y_{i}}\\
\end{split}
\end{equation*}
From lemma \ref{bq_fun}, we can rewrite the above relation as follows:
\begin{equation*}\label{eq: 1.1}
\begin{split}
{\sum_{i=r-1}^{n-rk}} \theta^{n-i}q^{ik}{\prod_{{j}=1}^i}\left(1-\theta q^{j-1}\right)  B_{q}^{k}(i,n-k-i,r-1).
\end{split}
\end{equation*}
\noindent
Thus, the proof is completed.
\end{proof}
\noindent
For $q=1$, from Theorem \ref{thmtype2q-nega}, the PMF of the Type \rom{2} negative binomial distribution of order $k$ in Bernoulli trials with success probability $\theta$ is obtained as follows:
\begin{corollary}
\label{cor:type1negativ}
The PMF $w^{(\rom{2})}(n;r,k;\theta)=P_{\theta}\left(W_{r,k}^{(\rom{2})}=n\right)$ is given by
\begin{equation*}\label{type2negativ}
\begin{split}
&P_{\theta}\left(W_{r,k}^{(\rom{2})}=n\right)=\\
&\left\{
  \begin{array}{ll}
    {\sum_{i=r-1}^{n-rk}} \theta^{n-i}(1-\theta)^{i}  B_{1}^{k}(i,n-k-i,r-1)  & \text{if $n> r(k+1)-1$,} \\
    \theta^{r(k+1)-1} & \text{if $n=r(k+1)-1$,} \\
    0, & \text{if n $< r(k+1)-1.$}\\
  \end{array}
\right.
\end{split}
\end{equation*}
\end{corollary}
\noindent

\noindent
\begin{remark}
{\rm
$G_{n,k}$ is a random variable related to $W_{r,k}^{(\rom{2})}$. It denotes the number of occurrences of a success run of length at least $k$ in a sequence of $n$ trials. Because the events $\left(G_{n,k}\geq r\right)$ and $\left(W_{r,k}^{(\rom{2})}\leq n\right)$ are equivalent, an alternative formula for the PMF of Type \rom{2} $q$-negative binomial distribution of order $k$, can be easily obtained using the dual relation between the binomial and negative binomial distribution of order $k$ is given by
\noindent
\begin{equation*}
P_{q,\theta}\left(G_{n,k}\geq r\right)=P_{q,\theta}\left(W_{r,k}^{(\rom{2})}\leq n\right).
\end{equation*}
\noindent
Consequently, the PMF $w_{q}^{(\rom{2})}(n;k,r;\theta)=P_{q,\theta}\Big(W_{r,k}^{(\rom{2})}=n\Big)$ is implicitly determined by
\begin{equation}
\label{alpmf:type2q-nega}
w_{q}^{(\rom{2})}(n;k,r;\theta)=\sum_{x=0}^{r-1}f_{q}^{(\rom{2})}(x;n-1,k;\theta)-f_{q}^{(\rom{2})}(x;n,k;\theta)\ \text{for}\ n\geq r(k+1)-1\ \text{and}\ r\geq 1,
\end{equation}
where the probabilities $f_{q}^{(\rom{2})}(x;n-1,k;\theta)=P_{q,\theta}(G_{n-1,k}=x)$ and $f_{q}^{(\rom{2})}(x;n,k;\theta)=P_{q,\theta}(G_{n,k}=x)$ are obtained following \citet{makri2016runs}.
\begin{equation}
\label{pmf:type2q-binomial}
f_{q}^{(\rom{2})}(x;n,k;\theta)={\sum_{i=0}^{n-xk}} \theta^{n-i}{\prod_{j=1}^i}(1-\theta q^{j-1})B_{q}^{k}(i+1,n-i,x)\ \text{for}\ x=0,1,\ldots,\left[\frac{n+1}{k+1}\right].
\end{equation}
Usually, the obtained expression \eqref{pmf:type2q-binomial} for $P_{q,\theta}\Big(W_{r,k}^{(\rom{2})}=n\Big)$ is computationally faster than that obtained using \eqref{alpmf:type2q-nega}.
}
\end{remark}


\section{Type \rom{3} $q$-negative binomial distribution of order $k$}
In this section, we study the Type \rom{3} $q$-negative binomial distribution of order $k$. Let us consider the waiting time for the $r$-th occurrence of the overlapping success run of length $k$. For $r\in N$ and $k\in N$, let $W_{r,k}^{(\rom{3})}$be the waiting time for the $r$-th appearance of the overlapping run of success of length $k$. We use the overlapping counting scheme (Type \rom{3} enumeration scheme proposed by \citet{ling1988binomial}); i.e., an uninterrupted sequence of $m\geq k$ successes preceded and followed by a failure or by the beginning or end of the sequence. The support (range set) of $W_{r,k}^{(\rom{3})}$, $\mathfrak{R}\Big(W_{r,k}^{(\rom{3})}\Big)$ is given by
\noindent
\begin{equation*}
\begin{split}
\mathfrak{R}\left(W_{r,k}^{(\rom{3})}\right)=\{k+r-1,k+r,\ldots\}.
\end{split}
\end{equation*}
\noindent
We list a useful definition and lemma for the proofs of Theorem in the sequel.
\noindent
\begin{definition}
For $0<q\leq1$, we define
\begin{equation*}\label{eq: 1.1}
\begin{split}
 C_{q}^{k}(r,s,t)=\sum_{\substack{y_{1},y_{2},\ldots,y_{r}}}   q^{y_{2}+2y_{3}+\cdots+(r-1) y_{r}}\\
\end{split}.
\end{equation*}
\noindent
Here, the summation is taken over all integers $y_1,\ldots,y_{r}$ satisfying
\noindent
\begin{equation*}\label{eq:1}
\begin{split}
y_{1}+y_{2}+\cdots+y_{r}=s,
\end{split}
\end{equation*}
\noindent
\begin{equation*}\label{eq:1}
\begin{split}
C(y_{1})+\cdots+C(y_{r})=t,\ \text{and}
\end{split}
\end{equation*}
\noindent
\begin{equation*}\label{eq:2}
\begin{split}
y_{j}\geq 0,\quad j=1,\ \ldots,\ r.
\end{split}
\end{equation*}
\noindent
Further,
\begin{equation*}\label{eq:2}
\begin{split}
C(j)=\left\{
  \begin{array}{ll}
    j-k+1, & \text{if}\ j\geq k,\\
    0, & \text{otherwise}
  \end{array}
\right.
\end{split}
\end{equation*}
\end{definition}
\noindent
Lemma \ref{cq_fun} gives a recurrence relation useful for the computation of $C_{q}^{k}(r,s,t)$.\\
\noindent
\begin{lemma}
\label{cq_fun}
{\rm [\citet{yalcin2013generalization}]} For $0<q\leq1$, $C_{q}^{k}(r,s,t)$ obeys the following recurrence relation,
\noindent
\begin{equation*}\label{eq: 1.1}
\begin{split}
&C_{q}^{k}(r,s,t)=\\&\left\{\begin{array}{ll}
\sum_{j=0}^{k-1} q^{(r-1)j} C_{q}^{k}(r-1,s-j,t)+\\
\sum_{j=k}^{s} q^{(r-1)j} C_{q}^{k}(r-1,s-j,t-j+k-1), &\text{if $r>1$, $s\geq 0$ and $t\geq 0$} \\
1,& \text{if $r=1$, $s\geq k$ and $t=s-k+1$} \\
&\text{or $r=1$, $0\leq s< k$ and $t=0$}, \\
0,&\text{otherwise.}\\
\end{array}
\right.
\end{split}
\end{equation*}
\end{lemma}

\begin{remark}
{\rm
We observe that $C_1^{k}(r,s,t)$ is the number of integer solutions $(y_{1},\ldots,y_{r})$ of
\noindent
\begin{equation*}\label{eq:1}
\begin{split}
y_{1}+y_{2}+\cdots+y_{r}=s,
\end{split}
\end{equation*}
\noindent
\begin{equation*}\label{eq:1}
\begin{split}
C(y_{1})+\cdots+C(y_{r})=t,\ \text{and}
\end{split}
\end{equation*}
\noindent
\begin{equation*}\label{eq:2}
\begin{split}
y_{j}\geq 0,\quad j=1,\ldots,r,
\end{split}
\end{equation*}
\noindent
where
\begin{equation*}\label{eq:2}
\begin{split}
C(j)=\left\{
  \begin{array}{ll}
    j-k+1, & \text{if}\ j\geq k,\\
    0, & \text{otherwise}
  \end{array}
\right.
\end{split}
\end{equation*}
\noindent
This means that the total number of arrangements of $s$ balls in $r$ distinguishable cells, yielding $t$ overlapping runs of balls of length $k$ is given by
\noindent
\begin{equation*}\label{eq: 1.1}
\begin{split}
C_{1}^{k}(r,s,t)=\sum_{a=1}^{\text{min}(r,t)}{r \choose a}{t-1 \choose a-1}S(r-a,k+1,s+r-t-ak),
\end{split}
\end{equation*}
\noindent
where 
\noindent
\begin{equation*}\label{eq: 1.1}
\begin{split}
S(a,\ b,\ c)=\sum_{j=0}^{a}(-1)^{j}{a \choose j}{c-j(b-1)-1 \choose a-1}
\end{split}
\end{equation*}
\noindent
(\citet{charalambides2002enumerative}).
}
\end{remark}


\noindent
The probability function of the Type \rom{3} $q$-negative binomial distribution of order $k$ is obtained from Theorem \ref{thmtype3q-nega}, and clearly,
\noindent
\begin{equation*}
P_{q,\theta}\left(W_{r,k}^{(\rom{3})}=n\right)=0\ \text{for}\ 0\leq n<k+r-1
\end{equation*}
\noindent
Hence, we focus on determining the PMF for $n\geq k+r-1$.
\noindent
\begin{theorem}
\label{thmtype3q-nega}
The PMF $w_{q}^{(\rom{3})}(n;r,k;\theta)=P_{q,\theta}\left(W_{r,k}^{(\rom{3})}=n\right)$ is given by
\begin{equation*}\label{eq: 1.1}
\begin{split}
&P_{q,\theta}\left(W_{r,k}^{(\rom{3})}=n\right)=\\
&\left\{
  \begin{array}{ll}
    \sum_{t=k}^{k+r-1}\sum_{i=\left[\frac{x-k-r}{k}\right]+1}^{n-k-r+1} \theta^{n-i}q^{it}{\prod_{{j}=1}^i}(1-\theta q^{j-1}) C_{q}^{k}(i,n-t-i,r-C(t))  & \text{if $n> k+r-1$,} \\
    \theta^{k+r-1} & \text{if $n=k+r-1$,} \\
    0, & \text{if}\ n < k+r-1.\\
  \end{array}
\right.
\end{split}
\end{equation*}
\end{theorem}

\begin{proof}
We first consider $w_{q}^{(\rom{3})}(k+r-1;r,k;\theta)$. Clearly, $w_{q}^{(\rom{3})}(k+r-1;r,k;\theta)=\big(\theta q^{0}\big)^{k+r-1}=\theta^{k+r-1}$. Hereinafter, $n > k+r-1.$ By the definition of $W_{r,k}^{(\rom{3})}$ every sequence of $n$ binary trials belonging to the event $W_{r,k}^{(\rom{3})}=n$ must end with $k$ successes, and the $r$-th overlapping success runs occur in the $n$-th trial. Let us consider $W_{r,k}^{(\rom{3})}=n$ end with $t$ successes. The event $W_{r,k}^{(\rom{3})}=n$ can be expressed as follows:
\begin{equation*}
\left\{W_{r,k}^{(\rom{3})}=n\right\}=\bigcup_{t=k}^{k+r-1}\left\{M_{n-t,k}=r-C(t)\ \wedge\ X_{n-t}=0\ \wedge\ X_{n-t+1}=\cdots=X_{n}=1\right\}.
\end{equation*}
\noindent
We partition the event $W_{r,k}^{(\rom{3})}=n$ into disjointed events given by $F_{n}=i,$ for $i=\left[\frac{x-k-r}{k}\right]+1,\ldots,n-(k+r-1).$ By adding the probabilities, we get
\noindent
\begin{equation*}\label{eq:bn}
\begin{split}
P_{q,\theta}\left(W_{r,k}^{(\rom{3})}=n\right)=\sum_{t=k}^{k+r-1}\sum_{i=\left[\frac{x-k-r}{k}\right]+1}^{n-(k+r-1)}P_{q,\theta}\Big(M_{n-t,k}&=r-C(t)\ \wedge\ X_{n-t}=0\ \wedge\\
&F_{n}=i\ \wedge\ X_{n-t+1}=\cdots=X_{n}=1\Big).\\
\end{split}
\end{equation*}
\noindent
If the number of zeroes in the first $n-t$ trials is equal to $i,$ that is, $F_{n-t}=i,$ then in each of the $(n-t+1)$ to $n$-th trials, the probability of success is
\noindent
\begin{equation*}\label{eq: kk}
\begin{split}
p_{n-t+1}=\cdots=p_{n}=\theta q^{i},
\end{split}
\end{equation*}
\noindent
which can be rewritten as follows.
\noindent
\begin{equation*}\label{eq:bn1}
\begin{split}
&P_{q,\theta}\left(W_{r,k}^{(\rom{3})}=n\right)\\
&=\sum_{t=k}^{k+r-1}\sum_{i=\left[\frac{x-k-r}{k}\right]+1}^{n-(k+r-1)}P_{q,\theta}\Big(M_{n-t,k}=r-C(t)\ \wedge\ X_{n-t}=0\ \wedge\ F_{n-t}=i\Big)\\
&\quad\quad\quad\quad\quad\quad\times P_{q,\theta}\Big(X_{n-t+1}=\cdots =X_{n}=1\mid F_{n-t}=i\Big)\\
&=\sum_{t=k}^{k+r-1}\sum_{i=\left[\frac{x-k-r}{k}\right]+1}^{n-(k+r-1)}P_{q,\theta}\Big(M_{n-t,k}=r-C(t)\ \wedge\ X_{n-t}=0\ \wedge\ F_{n-t}=i\Big)\Big(\theta q^{i}\Big)^{k}.\\
\end{split}
\end{equation*}
\noindent
An element of the event $\left\{W_{r,k}^{(\rom{3})}=n,\ F_{n}=i\right\}$is an ordered sequence consisting of $n-i$ successes and $i$ failures, such that the length of the success run is a nonnegative integer, has $r$ overlapping runs of success of length $k$, and ends with $t$ ($t=k,\ldots,k+r-1$) successes. The number of these sequences can be derived as follows: First, we distribute the $i$ failures, and $i$ failures form $i+1$ cells. Next, we distribute the $n-i-t$ successes into $i$ distinguishable cells as follows.
\noindent
\begin{equation*}\label{eq: 1.1}
\begin{split}
{\underbrace{1\ldots1}_{y_{1}}}0{\underbrace{1\ldots1}_{y_{2}}}0\ldots0{\underbrace{1\ldots1}_{y_{i-1}}}0{\underbrace{1\ldots1}_{y_{i}}}0{\boxed{{{\underbrace{1\ldots1}_{t}}}}}
\end{split}
\end{equation*}
\noindent
with $i$ 0s and $n-i$ 1s, where the length of the first one-run is $y_{1}$; the length of the second one-run is $y_{2}$,...; and the length of the $(i)$-th one-run is $y_{i}$. The probability of the event $\left\{W_{r,k}^{(\rom{3})}=n,\ F_{n}=i\right\}$ is given by
\noindent
\begin{equation*}\label{eq: 1.1}
\begin{split}
(\theta q^{0})^{y_{1}}(1-\theta q^{0})(\theta q^{1})^{y_{2}}(1-\theta q^{1})\cdots(\theta q^{i-1})^{y_{i}}(1-\theta q^{i-1})(\theta q^{i})^{t}.
\end{split}
\end{equation*}
\noindent
Simple exponentiation algebra arguments are used to simplify the above probability.
\noindent
\begin{equation*}\label{eq: 1.1}
\begin{split}
&\theta^{y_1+\cdots+y_{i}+t}q^{it}{\prod_{{j}=1}^i}\left(1-\theta q^{j-1}\right)q^{y_{2}+2y_{3}+\cdots+(i-1) y_{i}}\\
&\theta^{n-i}q^{it}{\prod_{{j}=1}^i}\left(1-\theta q^{j-1}\right)q^{y_{2}+2y_{3}+\cdots+(i-1) y_{i}}.
\end{split}
\end{equation*}
\noindent
However, $y_{j}$ are nonnegative integers such that ${y_{1}}+{y_{2}}+\cdots+{y_{i}}=n-t-i$ and
\noindent
\begin{equation*}\label{eq: 1.1}
\begin{split}
C(y_{1})+\cdots+C(y_{i})+C(t)=r,
\end{split}
\end{equation*}
\noindent
so that
\noindent
\begin{equation*}\label{eq: 1.1}
\begin{split}
&P_{q,\theta}\left(W_{r,k}^{(\rom{3})}=n,\ F_{n}=i\right)\\
&=\sum_{t=k}^{k+r-1}\theta^{n-i}q^{ik}{\prod_{{j}=1}^i}\left(1-\theta q^{j-1}\right)  {\mathop{\sum...\sum}_{\substack{y_{1}+y_{2}+\cdots+y_{i}=n-t-i\\ C(y_{1})+\cdots+C(y_{i})+C(t)=r
\\y_{1}\geq 0,\ldots,y_{i}\geq 0}}}   q^{y_{2}+2y_{3}+\cdots+(i-1) y_{i}}.\\
\end{split}
\end{equation*}
\noindent
Summing the above with respect to $i=\left[\frac{x-k-r}{k}\right]+1, ..., n-(k+r-1)$ results in\\
\noindent
\begin{equation*}\label{eq: 1.1}
\begin{split}
\sum_{t=k}^{k+r-1}\sum_{i=\left[\frac{x-k-r}{k}\right]+1}^{n-(k+r-1)}\theta^{n-i}q^{ik}{\prod_{{j}=1}^i}\left(1-\theta q^{j-1}\right) {\mathop{\sum...\sum}_{\substack{y_{1}+y_{2}+\cdots+y_{i}=n-t-i\\ C(y_{1})+\cdots+C(y_{i})+C(t)=r
\\y_{1}\geq 0,\ldots,y_{i}\geq 0}}}   q^{y_{2}+2y_{3}+\cdots+(i-1) y_{i}}\\
\end{split}
\end{equation*}
From lemma \ref{cq_fun}, we can rewrite the above as follows:
\begin{equation*}\label{eq: 1.1}
\begin{split}
\sum_{t=k}^{k+r-1}\sum_{i=\left[\frac{x-k-r}{k}\right]+1}^{n-(k+r-1)} \theta^{n-i}q^{it}{\prod_{{j}=1}^i}(1-\theta q^{j-1})  C_{q}^{k}(i,n-t-i,r-C(t)).
\end{split}
\end{equation*}
\noindent
Thus, the proof is completed.

%
%
%
%

\end{proof}
\noindent
For $q=1$, from Theorem \ref{thmtype3q-nega}, the PMF of the Type \rom{3} negative binomial distribution of order $k$ in Bernoulli trials with success probability $\theta$ is obtained as follows:
\noindent
\begin{corollary}
The PMF $w^{(\rom{3})}(n;r,k;\theta)=P_{\theta}\left(W_{r,k}^{(\rom{3})}=n\right)$ is given by
\begin{equation*}\label{eq: 1.1}
\begin{split}
P_{\theta}\Big(&W_{r,k}^{(\rom{3})}=n \Big)=\\
&\left\{
  \begin{array}{ll}
    \sum_{t=k}^{k+r-1}\sum_{i=\left[\frac{x-k-r}{k}\right]+1}^{n-(k+r-1)} \theta^{n-i}(1-\theta)^{i}  C_{1}^{k}(i,n-t-i,r-C(t))  & \text{if $n> k+r-1$,} \\
    \theta^{k+r-1} & \text{if $n=k+r-1$,} \\
    0, & \text{if}\ n < k+r-1.\\
  \end{array}
\right.
\end{split}
\end{equation*}

\end{corollary}

\begin{remark}
{\rm
$M_{n,k}$ is a random variable related to $W_{r,k}^{(\rom{3})}$. It denotes $M_{n,k}$ the number of occurrences of the overlapping success run of length $k$ in the sequence of $n$ trials. Because the events $\left(M_{n,k}\geq r\right)$ and $\left(W_{r,k}^{(\rom{3})}\leq n\right)$ are equivalent, an alternative formula for the PMF of Type $\rom{3}$ $q$-negative binomial distribution of order $k$, can be easily obtained, using the dual relation between the binomial and negative binomial distribution of order $k$ as follows:
\noindent
\begin{equation*}
P_{q,\theta}\left(M_{n,k}\geq r\right)=P_{q,\theta}\left(W_{r,k}^{(\rom{3})}\leq n\right).
\end{equation*}
\noindent
Consequently, the PMF $w_{q}^{(\rom{3})}(n;k,r;\theta)=P_{q,\theta}\left(W_{r,k}^{(\rom{3})}=n\right)$ is implicitly determined by
\begin{equation}
\label{alpmf:type3q-nega}
w_{q}^{(\rom{3})}(n;k,r;\theta)=\sum_{x=0}^{r-1}f_{q}^{(\rom{3})}(x;n-1,k;\theta)-f_{q}^{(\rom{3})}(x;n,k;\theta)\ \text{for}\ n\geq r(k+1)-1\ \text{and}\ r\geq 1,
\end{equation}
where the probabilities $f_{q}^{(\rom{3})}(x;n-1,k;\theta)=P_{q,\theta}(M_{n-1,k}=x)$ and $f_{q}^{(\rom{3})}(x;n,k;\theta)=P_{q,\theta}(M_{n,k}=x)$ are as obtained by \citet{yalcin2013generalization}:
\begin{equation}
\label{pmf:type3q-binomial}
f_{q}^{(\rom{3})}(x;n,k;\theta)={\sum_{i=0}^{n-k-x+1}} \theta^{n-i}{\prod_{j=1}^i}(1-\theta q^{j-1})C_{q}^{k}(i+1,n-i,x)\ \text{for}\ x=0,1,\ldots,\left[\frac{n+1}{k+1}\right].
\end{equation}
Usually, the obtained expression \eqref{pmf:type3q-binomial} for $P_{q,\theta}\left(W_{r,k}^{(\rom{3})}=n\right)$ is computationally faster than that obtained using \eqref{alpmf:type3q-nega}.
}
\end{remark}


\section{Type $\rom{4}$ $q$-negative binomial distribution of order $k$}


We examined the Type $\rom{4}$ $q$-negative binomial distribution of order $k$. First, we consider the waiting time for the $r$-th occurrence of the success run of length exactly $k$. For $r\in N$ and $k\in N$, let $W_{r,k}^{(\rom{4})}$be the waiting time for the $r$-th appearance of the run of successes of length exactly $k$. We use the length exactly $k$ counting scheme (Type \rom{4} enumeration scheme, as proposed by \citet{mood1940distribution}). In this scheme, a success run of length exactly $k$ is preceded and succeeded by failure or by nothing. The support (range set) of $W_{r,k}^{(\rom{4})}$, $\mathfrak{R}\left(W_{r,k}^{(\rom{4})}\right)$ is given by
\noindent
\begin{equation*}
\begin{split}
\mathfrak{R}\left(W_{r,k}^{(\rom{4})}\right)=\{r(k+1)-1,r(k+1),\ldots\}.
\end{split}
\end{equation*}
\noindent
We now present a useful definition and lemma for the proofs of Theorem \ref{thmtype4q-nega} in the sequel.
\noindent
\begin{definition}
For $0<q\leq1$, we define the polynomial \begin{equation*}\label{eq: 1.1} \begin{split}
 D_{q}^{k}(r,s,t)=&\sum_{y_{1},y_{2},\ldots,y_{r}}   q^{y_{2}+2y_{3}+\cdots+(r-1)y_{r}}
\end{split}.
\end{equation*}
\noindent
where the summation is preformed over all integers $y_1,\ldots,y_{r}$ that satisfy
\noindent
\begin{equation*}\label{eq:1}
\begin{split}
y_{1}+y_{2}+\cdots+y_{r}=s,
\end{split}
\end{equation*}
\noindent
\begin{equation*}\label{eq:1}
\begin{split}
\delta_{k,y_{1}}+\cdots+\delta_{k,y_{r}}=t,\ \text{and}
\end{split}
\end{equation*}
\noindent
\begin{equation*}\label{eq:2}
\begin{split}
y_{j}\geq 0,\quad j=1,\ \ldots,\ r.
\end{split}
\end{equation*}
\noindent
where
\begin{equation*}
  \delta_{i,j}=\left\{
               \begin{array}{ll}
                 1, & \text{if}\ i=j \\
                 0, & \text{if}\ i\neq j.
               \end{array}
                \right.
\end{equation*}
\end{definition}
\noindent
Now, we obtain a recurrence relation useful in the computation of $D_{q}^{k}(r,s,t)$.\\
\noindent
\begin{lemma}
\label{dq_fun}
{\rm [\cite{ohandjang2022}]}
For $0<q\leq1$, $D_{q}^{k}(r,s,t)$ obeys the following recurrence relation:
\noindent
\begin{equation*}\label{eq: 1.1}
\begin{split}
&D_{q}^{k}(r,s,t)\\
&=\left\{
  \begin{array}{ll}
    1, & \text{for}\ r=1,\ s=k,\ t=1 \\
    & \text{or}\ r=1,\ 0\leq s< k,\ t=0\\
        & \text{or}\ r=1,\ s> k,\ t=0\\
\sum_{j=0}^{k-1}q^{j(r-1)}D_{q}^{k}(r-1,s-j,t)+q^{k(r-1)}D_{q}^{k}(r-1,s-k,t-1)\\ +\sum_{j=k+1}^{s}q^{j(r-1)}D_{q}^{k}(r-1,s-j,t), & \text{for}\
r\geq 2,\ s\geq tk,\ t\leq r \\
    0, & \text{otherwise.}\\
  \end{array}
\right.
\end{split}
\end{equation*}
\end{lemma}

\begin{remark}
{\rm
We observe that $D_1^{k}(r,s,t)$ is the number of integer solutions $(y_{1},\ldots,y_{r})$ of
\noindent
\begin{equation*}\label{eq:1}
\begin{split}
y_{1}+y_{2}+\cdots+y_{r}=s,
\end{split}
\end{equation*}
\noindent
\begin{equation*}\label{eq:1}
\begin{split}
\delta_{k,y_{1}}+\cdots+\delta_{k,y_{r}}=t,\ \text{and}
\end{split}
\end{equation*}
\noindent
\begin{equation*}\label{eq:2}
\begin{split}
y_{j}\geq 0,\quad j=1,\ \ldots,\ r.
\end{split}
\end{equation*}
\noindent
where
\begin{equation*}
  \delta_{i,j}=\left\{
               \begin{array}{ll}
                 1, & \text{if}\ i=j \\
                 0, & \text{if}\ i\neq j.
               \end{array}
                \right.
\end{equation*}
\noindent
This means that the number of allocations of $s$ balls into $r$ cells so that each of exactly $t$
of them receives exactly equal to $k$ balls is given by
\begin{equation*}
\label{eq: 1.1}
\begin{split} D_{1}^{k}(r,s,t)={r \choose
t}A(s-tk,\ r-t,\ k),
\end{split}
\end{equation*}
where  $A(\alpha,r,k)=\sum_{j=0}^{[\alpha/k]}(-1)^{j}{r \choose
j}{\alpha-(k+1)j+r-1 \choose \alpha-jk}$ (see \citet{makri2007success}).
\noindent
}

\end{remark}

%
%

\noindent
The probability function of the Type $\rom{4}$ $q$-negative binomial distribution of order $k$ is obtained from the following theorem \ref{thmtype4q-nega}, and clearly,
\noindent
\begin{equation*}
P_{q,\theta}\left(W_{r,k}^{(\rom{4})}=n\right)=0\ \text{for}\ 0\leq n<r(k+1)-1
\end{equation*}
\noindent
Hence, we focus on determining the PMF for $n\geq r(k+1)-1$.
\noindent

\begin{theorem}
\label{thmtype4q-nega}
The PMF $w_{q}^{(\rom{4})}(n;r,k;\theta)=P_{q,\theta}\left(W_{r,k}^{(\rom{4})}=n\right)$ is given by
\begin{equation*}\label{eq: 1.1}
\begin{split}
&P_{q,\theta}\left(W_{r,k}^{(\rom{4})}=n\right)=\\
&\left\{
  \begin{array}{ll}
    {\sum_{i=r-1}^{n-rk}} \theta^{n-i}q^{ik}{\prod_{{j}=1}^i}(1-\theta q^{j-1})D_{q}^{k}(i,n-k-i,r-1)  & \text{if $n> r(k+1)-1$} \\
    \theta^{r(k+1)-1} & \text{if $n=r(k+1)-1$,} \\
    0, & \text{if $n< r(k+1)-1$.}\\
  \end{array}
\right.
\end{split}
\end{equation*}
\end{theorem}

\begin{proof}
We first examine $w_{q}^{(\rom{4})}(r(k+1)-1;r,k;\theta)$. Clearly, $w_{q}^{(\rom{4})}(r(k+1)-1;r,k;\theta)=\big(\theta q^{0}\big)^{r(k+1)-1}=\theta^{r(k+1)-1}$. Hereinafter, $n > r(k+1)-1.$ By the definition of $W_{r,k}^{(\rom{4})}$, every sequence of $n$ binary trials belonging to the event $W_{r,k}^{(\rom{4})}=n$ must end with $k$ successes. The event $W_{r,k}^{(\rom{4})}=n$ can be expressed as $$\left\{W_{r,k}^{(\rom{4})}=n\right\}=\left\{E_{n-k-1,k}=r-1\ \wedge\ X_{n-k}=0\ \wedge\ X_{n-k+1}=\cdots=X_{n}=1\right\}.$$
\noindent
We partition the event $W_{r,k}^{(\rom{4})}=n$ into disjointed events given by $F_{n}=i,$ for $i=r-1,\ldots,n-rk.$ Adding the probabilities, we get
\noindent
\begin{equation*}\label{eq:bn}
\begin{split}
P_{q,\theta}\Big(W_{r,k}^{(\rom{4})}=n\Big)=\sum_{i=r-1}^{n-rk}P_{q,\theta}\Big(E_{n-k,k}=r-1\ \wedge\ &X_{n-k}=0\ \wedge\ F_{n}=i\ \wedge\\
&X_{n-k+1}=\cdots=X_{n}=1\Big).\\
\end{split}
\end{equation*}
\noindent
If the number of zeroes in the first $n-k$ trials is equal to $i,$ that is, $F_{n-k}=i,$ then in each of the $(n-k+1)$ to $n$-th trials, the probability of success is
\noindent
\begin{equation*}\label{eq: kk}
\begin{split}
p_{n-k+1}=\cdots=p_{n}=\theta q^{i}.
\end{split}
\end{equation*}
\noindent
We can now rewrite this as follows:
\noindent
\begin{equation*}\label{eq:bn1}
\begin{split}
P_{q,\theta}\Big(W_{r,k}^{(\rom{4})}=n\Big)=\sum_{i=r-1}^{n-rk}&P_{q,\theta}\Big(E_{n-k,k}=r-1\ \wedge\ F_{n-k}=i\Big)\\
&\times P_{q,\theta}\Big(X_{n-k}=0\ \wedge\ X_{n-k+1}=\cdots =X_{n}=1\mid F_{n-k}=i\Big)\\
=\sum_{i=r-1}^{n-rk}&P_{q,\theta}\Big(E_{n-k,k}=r-1\ \wedge\ S_{n}=i\Big)\Big(1-\theta q^{i-1}\Big)\Big(\theta q^{i}\Big)^{k}.\\
\end{split}
\end{equation*}
\noindent
An element of the event $\left\{W_{r,k}^{(\rom{4})}=n,\ F_{n}=i\right\}$is an ordered sequence that consists of $n-i$ successes and $i$ failures such that the length of the success run is a nonnegative integer; $r$ nonoverlapping runs of success of length exactly $k$ end with exactly length $k$ successes. The number of these sequences can be derived as follows: First, we distribute the $i$ failures. Note that $i$ failures form $i+1$ cells. Next, we distribute the $n-i-k$ successes into $i$ distinguishable cells as follows:
\noindent
\begin{equation*}\label{eq: 1.1}
\begin{split}
\overbrace{{\underbrace{1\ldots1}_{y_{1}}}0{\underbrace{1\ldots1}_{y_{2}}}0\ldots0{\underbrace{1\ldots1}_{y_{i-1}}}0{\underbrace{1\ldots1}_{y_{i}}}}^{n-k-1}\overbrace{{\boxed{{0{\underbrace{1\ldots1}_{k}}}}}}^{k+1}
\end{split}
\end{equation*}
\noindent
with $i$ 0s and $n-i$ 1s. Here, the length of the first one-run is $y_{1}$, and the length of the second one-run is $y_{2}$,...; the length of the $(i)$-th one-run is $y_{i}$. The probability of the event $\left\{W_{r,k}^{(\rom{4})}=n,\ F_{n}=i\right\}$ is given by
\noindent
\begin{equation*}\label{eq: 1.1}
\begin{split}
(\theta q^{0})^{y_{1}}(1-\theta q^{0})(\theta q^{1})^{y_{2}}(1-\theta q^{1})\cdots(\theta q^{i-1})^{y_{i}}(1-\theta q^{i-1})(\theta q^{i})^{k}.
\end{split}
\end{equation*}
\noindent
Using simple exponentiation algebra arguments to simplify
\noindent
\begin{equation*}\label{eq: 1.1}
\begin{split}
\theta^{n-i}q^{ik}{\prod_{{j}=1}^i}(1-\theta q^{j-1})q^{y_{2}+2y_{3}+\cdots+(i-1) y_{i}}.
\end{split}
\end{equation*}
\noindent
However, $y_{j}$s are nonnegative integers such that ${y_{1}}+{y_{2}}+\cdots+{y_{i}}=n-k-i$ and
\noindent
\begin{equation*}\label{eq: 1.1}
\begin{split}
\delta_{k,y_{1}}+\delta_{k,y_{2}}+\cdots+\delta_{k,y_{i}}=r-1
\end{split}
\end{equation*}
\noindent
so that
\noindent
\begin{equation*}\label{eq: 1.1}
\begin{split}
&P_{q,\theta}\left(W_{r,k}^{(\rom{4})}=n,\ F_{n}=i\right)\\
&=\theta^{n-i}q^{ik}{\prod_{{j}=1}^i}\Big(1-\theta q^{j-1}\Big)  {\mathop{\sum...\sum}_{\substack{y_{1}+y_{2}+\cdots+y_{i}=n-k-i\\\delta_{k,y_{1}}+\delta_{k,y_{2}}+\cdots+\delta_{k,y_{i}}=r-1
\\y_{1}\geq 0,\ldots,y_{i}\geq 0}}}   q^{y_{2}+2y_{3}+\cdots+(i-1) y_{i}}.\\
\end{split}
\end{equation*}
\noindent
Summing the above with respect to $i=r-1, ..., n-rk$ results in\\
\noindent
\begin{equation*}\label{eq: 1.1}
\begin{split}
{\sum_{i=r-1}^{n-rk}} \theta^{n-i}q^{ik}{\prod_{{j}=1}^i}\Big(1-\theta q^{j-1}\Big)  {\mathop{\sum...\sum}_{\substack{y_{1}+y_{2}+\cdots+y_{i}=n-k-i\\\delta_{k,y_{1}}+\delta_{k,y_{2}}+\cdots+\delta_{k,y_{i}}=r-1
\\y_{1}\geq 0,\ldots,y_{i}\geq 0}}}   q^{y_{2}+2y_{3}+\cdots+(i-1) y_{i}}\\
\end{split}
\end{equation*}
\noindent
From lemma \ref{dq_fun}, we can rewrite the above as follows:
\noindent
\begin{equation*}\label{eq: 1.1}
\begin{split}
{\sum_{i=r-1}^{n-rk}} \theta^{n-i}q^{ik}{\prod_{{j}=1}^i}\Big(1-\theta q^{j-1}\Big)  D_{q}^{k}(i,n-k-i,r-1).
\end{split}
\end{equation*}
\noindent
Thus, the proof is completed.

%
%
%
%

\end{proof}
\noindent
For $q=1$, from Theorem \ref{thmtype4q-nega}, the PMF of the Type \rom{4} negative binomial distribution of order $k$ in Bernoulli trials with success probability $\theta$ is obtained as follows:

\begin{corollary}
The PMF $w^{(\rom{4})}(n;r,k;\theta)=P_{\theta}\left(W_{r,k}^{(\rom{4})}=n\right)$ is given by

\begin{equation*}\label{eq: 1.1}
\begin{split}
P_{\theta}&\left(W_{r,k}^{(\rom{4})}=n\right)=\\
&\left\{
  \begin{array}{ll}
    {\sum_{i=r-1}^{n-rk}} \theta^{n-i}(1-\theta )^{i}  D_{1}^{k}(i,n-k-i,r-1)  & \text{if $n>r(k+1)-1$} \\
    \theta^{r(k+1)-1} & \text{if $n=r(k+1)-1$} \\
    0, & \text{if $n< r(k+1)-1$.}\\
  \end{array}
\right.
\end{split}
\end{equation*}
\end{corollary}

\section{$q$-negative binomial distribution of order $k$ in the $\ell$-overlapping case}
We examined the $q$-negative binomial distribution of order $k$ in the $\ell$-overlapping case. Let us consider the waiting time for the $r$-th occurrence of the $\ell$-overlapping success run of length exactly $k$. For $r\in N$ and $k\in N$, let $W_{r,k\ell}$ be the waiting time for the $r$-th appearance of the $\ell$-overlapping run of successes of length $k$. We will use the $\ell$-overlapping counting scheme, i.e., a success run of length  $k$ each of which may have an overlapping (common) part of length at most $\ell$ ($\ell=0,1,\ldots,k-1$) with the previous run of success of length $k$ that has already been enumerated. The support (range set) of $W_{r,k,\ell}$, $\mathfrak{R}(W_{r,k,\ell})$ is given by
\noindent
\begin{equation*}
\begin{split}
\mathfrak{R}(W_{r,k,\ell})=\{\ell+r(k-\ell),\ell+r(k-\ell)+1,\ldots\}.
\end{split}
\end{equation*}
\noindent
We now present a useful definition and lemma for the proofs of the theorem \ref{pmfl-overq-nega} in the sequel.
\noindent
\begin{definition}
For $0<q\leq1$, define
\begin{equation*}\label{eq: 1.1}
\begin{split}
 E_{q}^{k,l}(r,s,t)=\sum_{\substack{y_{1},y_{2},\ldots,y_{r}}}   q^{y_{2}+2y_{3}+\cdots+(r-1) y_{r}}\\
\end{split}.
\end{equation*}

\noindent
where the summation is performed over all integers $y_1,\ldots,y_{r}$ satisfying
\noindent

\begin{equation*}\label{eq:1}
\begin{split}
y_{1}+y_{2}+\cdots+y_{r}=s,
\end{split}
\end{equation*}
\noindent
\begin{equation*}\label{eq:1}
\begin{split}
D(y_{i})+\cdots+D(y_{r})=t,\ \text{and}
\end{split}
\end{equation*}
\noindent
\begin{equation*}\label{eq:2}
\begin{split}
y_{j}\geq 0,\quad j=1,\ \ldots,\ r.
\end{split}
\end{equation*}
\noindent
Here,
\begin{equation*}\label{eq:2}
\begin{split}
D(j)=\left\{
  \begin{array}{ll}
    \left[\frac{j-l}{k-l}\right], & \text{if}\ j\geq k,\\
    0, & \text{otherwise}
  \end{array}
\right.
\end{split}
\end{equation*}
\end{definition}
\noindent
Now, we present a recurrence relation useful for computing $E_{q}^{k,l}(r,s,t)$.\\
\noindent
\begin{lemma}
\label{eq_fun}
{\rm [\citet{kinaci2016number}]} For $0<q\leq1$, $E_{q}^{k,l}(r,s,t)$ obeys the following recurrence relation:
\noindent
\begin{equation*}\label{eq: 1.1}
\begin{split}
&E_{q}^{k,l}(r,s,t)=\\&\left\{\begin{array}{ll}
\sum_{j=0}^{k-1} q^{(r-1)j} E_{q}^{k,l}(r-1,s-j,t)+\\
\sum_{j=k}^{s} q^{(r-1)j} E_{q}^{k,l}(r-1,s-j,t-D(j)), & \text{if $r>1$, $s\geq 0$ and $t\geq 0$} \\
1,& \text{ if $r=1$, $s\geq k$ and $t=\frac{s-l}{k-l}$} \\
& \text{ or $r=1$, $0\leq s< k$ and $t=0$}, \\
0,& \text{otherwise.}\\
\end{array}
\right.
\end{split}
\end{equation*}
\end{lemma}

\begin{remark}
{\rm
Note that $E_1^{k,l}(r,s,t)$ is the number of integer solutions $(y_{1},\ldots,y_{r})$ of
\noindent

\begin{equation*}\label{eq:1}
\begin{split}
y_{1}+y_{2}+\cdots+y_{r}=s,
\end{split}
\end{equation*}
\noindent
\begin{equation*}\label{eq:1}
\begin{split}
D(y_{1})+\cdots+D(y_{r})=t,\ \text{and}
\end{split}
\end{equation*}
\noindent
\begin{equation*}\label{eq:2}
\begin{split}
y_{j}\geq 0,\quad j=1,\ \ldots,\ r.
\end{split}
\end{equation*}
\noindent
Here,
\begin{equation*}\label{eq:2}
\begin{split}
D(j)=\left\{
  \begin{array}{ll}
    \left[\frac{j-l}{k-l}\right], & \text{if}\ j\geq k,\\
    0, & \text{otherwise}
  \end{array}
\right.
\end{split}
\end{equation*}
\noindent
indicating that the total number of arrangements of $s$ balls in $r$ distinguishable cells, yielding $t$ $l$-overlapping runs of balls of length $k$ is given by
\noindent
\begin{equation*}\label{eq: 1.1}
\begin{split}
E_{1}^{k,l}(r,s,t)=\sum_{a=1}^{\text{min}(r,t)}{r \choose a}{t-1 \choose a-1}C(s-al-(k-l)s;a,r-a;k-l-1,k-1).
\end{split}
\end{equation*}
\noindent
Here, $C(\alpha ;i,r-i;m-a,n-1)$ denotes the total number of integer solutions arrangements of $\alpha$ indistinguishable balls into $r$ distinguishable cells, $i$ of which have capacity $m-1$, and each of the rest of $r-i$ has capacity $n-1$. This number can be expressed as follows:
\noindent
\begin{equation*}\label{eq: 1.1}
\begin{split}
C(\alpha;i,r-i;m-a,n-1)=\sum_{j_1=0}^{[\alpha/m]}\sum_{j_2=0}^{[(\alpha-mj_1)/n]}(-1)^{j_1+j_2}{i \choose j_1}{r-i\choose j_2}{\alpha-mj_1-nj_2+r-1 \choose r-1}
\end{split}
\end{equation*}
\noindent
(See \citet{makri2007polya}).
}
\end{remark}
\noindent
The probability function of the $q$-negative binomial distribution of order $k$ in the $\ell$-overlapping case is obtained from the following theorem \ref{pmfl-overq-nega} as shown below:
\noindent
\begin{equation*}
P_{q,\theta}\left(W_{r,k,\ell}=n\right)=0\ \text{for}\ 0\leq n<l+r(k-l),
\end{equation*}
\noindent
and hence, we focus on determining the PMF for $n\geq l+r(k-l)$.
\noindent
\begin{theorem}
\label{pmfl-overq-nega}
The PMF $w_{q,\ell}(n;r,k;\theta)=P_{q,\theta}\Big(W_{r,k,\ell}=n\Big)$ is given by
\begin{equation*}\label{eq: 1.1}
\begin{split}
&P_{q,\theta}\Big(W_{r,k,\ell}=n\Big)=\\
&\left\{
  \begin{array}{ll}
   \sum_{t=k}^{r(k-l)+l}  \sum_{i=\left[\frac{x-l-1-r(k-l)}{k}\right]+1}^{n-r(k-l)-l}\theta^{n-i}q^{it}{\prod_{{j}=1}^i}(1-\theta q^{j-1})\times\\
      E_{q}^{k,l}(i,n-t-i,r-D(t)),  & \text{if $n\geq l+r(k-l)$} \\
    \theta^{l+r(k-l)}, & \text{if $n=l+r(k-l)$} \\
    0, & \text{if}\ n < l+r(k-l).\\
  \end{array}
\right.
\end{split}
\end{equation*}
\end{theorem}
\begin{proof}
We first examine $w_{q}^{(\ell)}(l+r(k-l);r,k;\theta)$. Clearly, $w_{q}^{(\ell)}(l+r(k-l);r,k;\theta)=\big(\theta q^{0}\big)^{l+r(k-l)}=\theta^{l+r(k-l)}$. Hereafter, we assume $n > l+r(k-l).$ From the definition of $W_{r,k,\ell}$, every sequence of $n$ binary trials belonging to the event $W_{r,k,\ell}=n$ must end with $k$ successes, and the $r$-th $\ell$-overlapping success runs occur in the $n$th trial. Let us consider $W_{r,k,\ell}=n$ end with $t$ successes. The event $W_{r,k,\ell}=n$ can be expressed as $$\left\{W_{r,k,\ell}=n\right\}=\bigcup_{t=k}^{r(k-l)+l}\left\{N_{n-t,k,\ell}=r-D(t)\ \wedge\ X_{n-t}=0\ \wedge\ X_{n-t+1}=\cdots=X_{n}=1\right\}.$$
\noindent
We partition the event $W_{r,k,\ell}=n$ into disjointed events given by $F_{n}=i,$ for $i=\left[\frac{x-l-1-r(k-l)}{k}\right]+1,\ldots,n-r(k-l)-l.$ By adding the probabilities, we get
\noindent
\begin{equation*}\label{eq:bn}
\begin{split}
P_{q,\theta}\left(W_{r,k,\ell}=n\right)=\sum_{t=k}^{r(k-l)+l} \sum_{i=\left[\frac{x-l-1-r(k-l)}{k}\right]+1}^{n-l-r(k-l)}P_{q,\theta}\Big(N_{n-t,k,\ell}&=r-D(t)\ \wedge\ X_{n-t}=0\ \wedge\\
&F_{n}=i\ \wedge\ X_{n-t+1}=\cdots=X_{n}=1\Big).\\
\end{split}
\end{equation*}
\noindent
If the number of zeroes in the first $n-t$ trials is equal to $i,$ that is, $F_{n-t}=i,$ then in each of the $(n-t+1)$ to $n$-th trials, the probability of success is
\noindent
\begin{equation*}\label{eq: kk}
\begin{split}
p_{n-t+1}=\cdots=p_{n}=\theta q^{i}.
\end{split}
\end{equation*}
\noindent
The above expression can be rewritten as follows.
\noindent
\begin{equation*}\label{eq:bn1}
\begin{split}
&P_{q,\theta}\left(W_{r,k,\ell}=n\right)\\
&=\sum_{t=k}^{r(k-l)+l} \sum_{i=\left[\frac{x-l-1-r(k-l)}{k}\right]+1}^{n-l-r(k-l)}P_{q,\theta}\Big(N_{n-t,k,\ell}=r-D(t)\ \wedge\ X_{n-t}=0\ \wedge\ F_{n-t}=i\Big)\\
&\quad\quad\quad\quad\quad\quad\quad\quad\quad\times P_{q,\theta}\Big(X_{n-t+1}=\cdots =X_{n}=1\mid F_{n-t}=i\Big)\\
&=\sum_{t=k}^{r(k-l)+l} \sum_{i=\left[\frac{x-l-1-r(k-l)}{k}\right]+1}^{n-l-r(k-l)}P_{q,\theta}\Big(N_{n-t,k,\ell}=r-D(t)\ \wedge\ X_{n-t}=0\ \wedge\ F_{n-t}=i\Big)\Big(\theta q^{i}\Big)^{k}.\\
\end{split}
\end{equation*}
\noindent
An element of the event $\Big\{W_{r,k,\ell}=n,\ F_{n}=i\Big\}$ is an ordered sequence that comprising $n-i$ successes and $i$ failures such that the length of the success run is nonnegative integer with $r$ overlapping runs of success of length $k$ and end with $t (t=k,\ldots,r(k-\ell)+\ell)$ successes. The number of these sequences can be derived as follows: First, we distribute the $i$ failures; $i$ failures form $i+1$ cells. Next, we distribute the $n-i-t$ successes into $i$ distinguishable cells as follows.
\noindent
\begin{equation*}\label{eq: 1.1}
\begin{split}
{\underbrace{1\ldots1}_{y_{1}}}0{\underbrace{1\ldots1}_{y_{2}}}0\ldots0{\underbrace{1\ldots1}_{y_{i-1}}}0{\underbrace{1\ldots1}_{y_{i}}}0{\boxed{{{\underbrace{1\ldots1}_{t}}}}}
\end{split}
\end{equation*}
\noindent
with $i$ 0s and $n-i$ 1s, where the length of the first one-run is $y_{1}$, the length of the second one-run is $y_{2}$,..., the length of the $(i)$-th one-run is $y_{i}$. The probability of the event $\left\{W_{r,k,\ell}=n,\ F_{n}=i\right\}$ is given by
\noindent
\begin{equation*}\label{eq: 1.1}
\begin{split}
(\theta q^{0})^{y_{1}}(1-\theta q^{0})(\theta q^{1})^{y_{2}}\left(1-\theta q^{1}\right)\cdots(\theta q^{i-1})^{y_{i}}(1-\theta q^{i-1})\left(\theta q^{i}\right)^{t}.
\end{split}
\end{equation*}
\noindent
We use simple exponentiation algebra arguments to simplify
\noindent
\begin{equation*}\label{eq: 1.1}
\begin{split}
&\theta^{y_1+\cdots+y_{i}+t}q^{it}{\prod_{{j}=1}^i}\left(1-\theta q^{j-1}\right)q^{y_{2}+2y_{3}+\cdots+(i-1) y_{i}}\\
&\theta^{n-i}q^{it}{\prod_{{j}=1}^i}\left(1-\theta q^{j-1}\right)q^{y_{2}+2y_{3}+\cdots+(i-1) y_{i}}.
\end{split}
\end{equation*}
\noindent
However, $y_{j}$s are nonnegative integers such that ${y_{1}}+{y_{2}}+\cdots+{y_{i}}=n-t-i$ and
\noindent
\begin{equation*}\label{eq: 1.1}
\begin{split}
D(y_{1})+\cdots+D(y_{i})+D(t)=r
\end{split}
\end{equation*}
\noindent
so that
\noindent
\begin{equation*}\label{eq: 1.1}
\begin{split}
&P_{q,\theta}\left(W_{r,k,\ell}=n,\ S_{n}=i\right)\\
&=\sum_{t=k}^{r(k-l)+l}\theta^{n-i}q^{it}{\prod_{{j}=1}^i}\left(1-\theta q^{j-1}\right)  {\mathop{\sum...\sum}_{\substack{y_{1}+y_{2}+\cdots+y_{i}=n-t-i\\ D(y_{1})+\cdots+D(y_{i})+D(t)=r
\\y_{1}\geq 0,\ldots,y_{i}\geq 0}}}   q^{y_{2}+2y_{3}+\cdots+(i-1) y_{i}}.\\
\end{split}
\end{equation*}
\noindent
Summing the above with respect to $i=\left[\frac{x-l-1-r(k-l)}{k}\right]+1, ...,n-l-r(k-l)$, we get\\
\noindent
\begin{equation*}\label{eq: 1.1}
\begin{split}
\sum_{t=k}^{r(k-l)+l}\sum_{i=\left[\frac{x-l-1-r(k-l)}{k}\right]+1}^{n-l-r(k-l)}\theta^{n-i}q^{it}{\prod_{{j}=1}^i}\left(1-\theta q^{j-1}\right)  {\mathop{\sum...\sum}_{\substack{y_{1}+y_{2}+\cdots+y_{i}=n-t-i\\ D(y_{1})+\cdots+D(y_{i})+D(t)=r
\\y_{1}\geq 0,\ldots,y_{i}\geq 0}}}   q^{y_{2}+2y_{3}+\cdots+(i-1) y_{i}}\\
\end{split}
\end{equation*}
From lemma \ref{eq_fun}, we can rewrite as follows:
\begin{equation*}\label{eq: 1.1}
\begin{split}
\sum_{t=k}^{r(k-l)+l}\sum_{i=\left[\frac{x-l-1-r(k-l)}{k}\right]+1}^{n-l-r(k-l)} \theta^{n-i}q^{it}{\prod_{{j}=1}^i}(1-\theta q^{j-1})  E_{q}^{k,l}(i,n-t-i,r-D(t)).
\end{split}
\end{equation*}
\noindent
Thus, the proof is completed.

%
%
%
%

\end{proof}
\noindent
For $q=1$, from Theorem \ref{pmfl-overq-nega}, the PMF of the negative binomial distribution of order $k$ for $\ell$-overlapping success runs of length $k$ in Bernoulli trials with success probability $\theta$ is obtained as follows:
\begin{corollary}
The PMF $w^{(\ell)}(n;r,k;\theta)=P_{\theta}\left(W_{r,k,\ell}=n\right)$ is given by
\begin{equation*}\label{eq: 1.1}
\begin{split}
&P_{\theta}\Big(W_{r,k,\ell}=n\Big)=\\
&\left\{
  \begin{array}{ll}
    \sum_{t=k}^{r(k-l)+l}\sum_{i=\left[\frac{x-l-1-r(k-l)}{k}\right]+1}^{n-r(k-l)-l} \theta^{n-i}(1-\theta )^{i}      E_{1}^{k,l}(i,n-t-i,r-D(t))  & \text{if $n> l+r(k-l)$,} \\
    \theta^{l+r(k-l)} & \text{if $n=l+r(k-l)$,} \\
    0, & \text{if}\ n < l+r(k-l).\\
  \end{array}
\right.
\end{split}
\end{equation*}

\end{corollary}

\begin{remark}
{\rm
$N_{n,k,\ell}$ is a random variable related to $W_{r,k,\ell}$, and it denotes the number of occurrences of success run of length $k$ in the sequence of $n$ trials. Because the events $\left(N_{n,k,\ell}\geq r\right)$ and $\left(W_{r,k,\ell}\leq n\right)$ are equivalent, an alternative formula for the PMF of $q$-negative binomial distribution of order $k$ in the $\ell$-overlapping case, can be easily obtained using the dual relation between the binomial and negative binomial distribution of order $k$ in the $\ell$-overlapping case as follows:
\noindent
\begin{equation*}
P_{q,\theta}\left(N_{n,k,\ell}\geq r\right)=P_{q,\theta}\left(W_{r,k,\ell}\leq n\right).
\end{equation*}
\noindent
Consequently, the PMF $w_{q}^{(\ell)}(n;k,r,\ell;\theta)=P_{q,\theta}(W_{r,k,\ell}=n)$ is implicitly determined by
\begin{equation}
\label{alpmf:;loverq-nega}
w_{q}^{\ell}(n;k,r,\ell;\theta)=\sum_{x=0}^{r-1}f_{q}^{(\ell)}(x;n-1,k,\ell;\theta)-f_{q}^{(\ell)}(x;n,k,\ell;\theta)\ \text{for}\ n\geq \ell+r(k-\ell)\ \text{and}\ r\geq 1,
\end{equation}
where the probabilities $f_{q}^{\ell}(x;n-1,k,\ell;\theta)=P_{q,\theta}(N_{n-1,k,\ell}=x)$ and $f_{q}^{\ell}(x;n,k,\ell;\theta)=P_{q,\theta}(N_{n,k,\ell}=x)$ already obtained by \citet{kinaci2016number} as follows:
\begin{equation}
\label{pmf:loverq-binomial}
f_{q}^{(\ell)}(x;n,k,\ell;\theta)={\sum_{i=0}^{\nu(x)}} \theta^{n-i}{\prod_{j=1}^i}(1-\theta q^{j-1})E_{q}^{k,\ell}(i+1,n-i,x)\ \text{for}\ x=0,1,\ldots,\left[\frac{n-\ell}{k-\ell}\right],
\end{equation}
where $\nu(x)=\left\{
  \begin{array}{ll}
    n, & \hbox{if $x=0$} \\
    n-(x(k-\ell)+\ell), & \hbox{otherwise.}
  \end{array}
\right.$\\
Usually, the obtained expression \eqref{pmf:loverq-binomial} for $P_{q,\theta}(W_{r,k,\ell}=n)$ is computationally faster than that obtained using \eqref{alpmf:;loverq-nega}.
}
\end{remark}

\section{Discussion}
This study introduces and investigates $q-NB_{k}^{(\rom{1})}(n,\theta)$, $q-NB_{k}^{(\rom{2})}(n,\theta)$, $q-NB_{k}^{(\rom{3})}(n,\theta)$, $q-NB_{k}^{(\rom{4})}(n,\theta)$, and $q-NB_{k,\ell}(r,\theta)$. We have an developed expression for the PMF of random variables $W_{r,k}^{(\rom{1})}$, $W_{r,k}^{(\rom{2})}$, $W_{r,k}^{(\rom{3})}$, $W_{r,k}^{(\rom{4})}$ and $W_{r,k,\ell}$ that follow $q-NB_{k}^{(\rom{1})}(n,\theta)$, $q-NB_{k}^{(\rom{2})}(n,\theta)$, $q-NB_{k}^{(\rom{3})}(n,\theta)$, $q-NB_{k}^{(\rom{4})}(n,\theta)$, and $q-NB_{k,\ell}(r,\theta)$, respectively. The random variable $N_{n}$ is closely related to the random variable $W_{r,k}$ (\citet{feller1968introduction}). We have the following dual relationship $N_{n}^{(a)}<r$ if and only if $W_{r,k}^{(a)}>n$ for $a=\rom{1},\rom{2},\rom{3}$. Now, the probability function of the $q$-binomial distribution of order $k$ can be easily derived using the dual relationship between the binomial and negative binomial distributions of order $k$:
\begin{equation*}
\begin{split}
&P_{q,\theta}\left(N_{n}^{(\rom{1})}<r\right)=P_{q,\theta}\left(W_{r,k}^{(\rom{1})}>n\right),\ P_{q,\theta}\left(N_{n}^{(\rom{2})}<r\right)=P_{q,\theta}\left(W_{r,k}^{(\rom{2})}>n\right)\\
&P_{q,\theta}\left(N_{n}^{(\rom{3})}<r\right)=P_{q,\theta}\left(W_{r,k}^{(\rom{3})}>n\right),\ P_{q,\theta}\left(N_{n,\ell}<r\right)=P_{q,\theta}\left(W_{r,k,\ell}>n\right).
\end{split}
\end{equation*}
\noindent
However, the aforementioned dual relationship cannot be considered by the Type $\rom{4}$ enumeration scheme. Because  $W_{r,k}^{(\rom{4})}>n$ implies $N_{n}^{(\rom{4})}<r$, instead of an iff relationship. A straightforward and simple explicit form of PMF $P_{q,\theta}(E_{n,k})$ was studied by \citet{oh2022success}, while and PMF $P_{\theta}(E_{n,k})$ was studied by \citet{makri2007success} as well as \citet{makri2011success}.

\section*{Acknowledgements.} The author would like to thank Dr. habil. Tommy Rene Jensen whose comments led to significant improvements in this manuscript.

\section*{Funding.} This research was partially supported by Basic Science Research Program through the National Research Foundation of Korea(NRF) funded by the Ministry of Education(NRF-2019R1A6A3A01090663, NRF-2020R1I1A1A01057479) and partially supported under the framework of international cooperation program managed by the National Research Foundation of Korea(NRF-2022K2A9A1A01098016).

%
%
%

\bibliography{biblio}
\addcontentsline{toc}{section}{References}
\bibliographystyle{apalike}

\end{document}